\documentclass[11pt]{amsart}
\pdfoutput=1

\usepackage{amsxtra}
\usepackage{amsfonts}
\usepackage[latin1]{inputenc}
\usepackage{graphicx}
\usepackage{amsmath,amssymb,latexsym,amsthm}
\usepackage[all]{xy}
\newtheorem{theorem}{Theorem}[section]
\newtheorem{thm}[theorem]{Theorem}
\newtheorem{lemma}[theorem]{Lemma}

\newtheorem{cor}[theorem]{Corollary}

\newtheorem{prop}[theorem]{Proposition}

\newtheorem{conj}[theorem]{Conjecture}
\newtheorem{remark}[theorem]{Remark}

\newtheorem*{kummer}{Kummer's Theorem}

\CompileMatrices

\DeclareMathOperator{\Ext}{Ext} \DeclareMathOperator{\THH}{THH}
\DeclareMathOperator{\THC}{THH} \DeclareMathOperator{\Tor}{Tor}
\DeclareMathOperator{\rTHH}{\overline{THH}}

\newcommand{\A}{\mathcal{A}}

\newcommand{\F}{\mathbb{F}}

\newcommand{\Q}{\mathbb{Q}}

\newcommand{\Z}{\mathbb{Z}}
\newcommand{\sma}{\wedge}

\newcommand{\TC}{TC}

\begin{document}

\title{Topological Hochschild Homology of $\ell$ and $ko$}

\author{Vigleik Angeltveit, Michael Hill and Tyler Lawson}

\begin{abstract}
We calculate the integral homotopy groups of $\THH(\ell)$ at any prime and of $\THH(ko)$ at $p=2$, where $\ell$ is the Adams summand of the connective complex $p$-local $K$-theory spectrum and $ko$ is the connective real $K$-theory spectrum.
\end{abstract}

\maketitle

\section{Introduction}\label{sec:Intro}
\subsection{Motivation}
Topological Hochschild homology is a generalization of Hochschild homology to the context of structured ring spectra.  In analogy with Hochschild homology, it helps classifying deformations and extensions of structured ring spectra.

In addition, the work of numerous authors on the cyclotomic trace now gives machinery allowing the computation of the algebraic $K$-theory of connective ring spectra $R$ \cite{BHM, HeMa97In}.  The first necessary input to these computations is the topological Hochschild homology $\THH(R)$.

After localization at a fixed prime $p$, the connective complex $K$-theory spectrum $ku$ has a summand $\ell$, known as the Adams summand, and when $p=2$, $\ell=ku_{(2)}$.  McClure and Staffeldt carried out the computation of the mod $p$ homotopy of $\THH(\ell)$ at primes $p \geq 3$ \cite{McSt}.  Their method was to first compute the $E_2$-term of the Adams spectral sequence converging to the mod $p$ homotopy, and then use knowledge about the $K(1)$-homology to obtain necessary information about the differentials. These methods lead to difficulties at the prime $2$ because the mod $2$ Moore spectrum is not a ring spectrum. The computations were extended by Rognes and the first author to the case $p=2$ \cite{AnRo}.

Ausoni and Rognes have computed the $V(1)$-homotopy of the topological cyclic homology $\TC(\ell)$ and of the algebraic $K$-theory $K(\ell)$ \cite{AuRo}, beginning by computing the homotopy of $V(1) \sma \THH(\ell)$.  This leads to problems at the prime $3$, where $V(1)$ is not a ring spectrum, and at the prime $2$, where $V(1)$ does not exist, and so their computations were only valid for $p \geq 5$.  Ausoni later extended their $\THH$ computations to $p=3$ and to computations for $ku$, rather than for the Adams summand
\cite{Au05}.

McClure and Staffeldt stated in \cite{McSt} their intent to continue their project: ``In the sequel we will investigate the integral homotopy groups of $\THH(\ell)$ using our present results as a starting point.''  While extensive computations were carried out,
this sequel never appeared.

The aim of this paper is to use some of the recent advances in structured ring spectra to both simplify the previous computations of $\THH(\ell)$, in some cases removing any restrictions on the prime, and to exhibit a complete integral computation of $\THH_\ast(\ell)$ as an $\ell_\ast$-module.  One finds that there is a weak equivalence between the spectrum $V(1) \sma \THH(\ell)$, for $p$ odd, and $\THH(\ell;H\F_p)$, and the latter is the realization of a simplicial commutative $H\F_p$-algebra.  It should be noted that this method does not simplify the computations of topological cyclic homology and algebraic $K$-theory, as neither of these spectra inherit the structure of module spectra over $\ell$.

In addition, there are Bockstein spectral sequences for computing the homotopy of $\THH(\ell;H\Z_{(p)})$ and that of $\THH(\ell;\ell/p)$ from $\THH_*(\ell;H\F_p)$, and for computing $\THH_\ast(\ell)$ from $\THH_\ast(\ell;\ell/p)$ or from $\THH_\ast(\ell;H\Z_{(p)})$. It happens that the integral computation of the homotopy groups of $\THH(\ell)$ is determined by the requirement that the two Bockstein spectral
sequences converging to $\THH_\ast(\ell)$ agree.  We highly recommend that the reader experiment with this method at $p=2$ to gain insight into the final result.

Similarly, this ``dueling Bockstein'' method can be used to give a complete computation of $\THH_\ast(ko;ku)$ $2$-locally, and the results are strikingly similar to the $2$-local computation of $\THH_\ast(ku)$.  There is then a final $\eta$-Bockstein spectral sequence computing the $2$-local homotopy of $\THH(ko)$ which is directly computable. One, perhaps unexpected, result of this computation is that $\eta^2$ acts by zero on the homotopy of $\rTHH(ko)$, the summand of $\THH(ko)$ complementary to $ko$.

\subsection{Organization} We begin in Section~\ref{sec:prelim} by summarizing the key tools we will need to start the computations and stating our main results. In Section~\ref{sec:BSS1and2}, we run the first two Bockstein spectral sequences. This provides the necessary input to allow us to run the last two spectral sequences. In Section~\ref{sec:SSB3}, we analyze the third Bockstein spectral sequence, using it to get information about the possible structure of the homotopy groups. Section~\ref{sec:THCku} is a brief digression into topological Hochschild cohomology, and in it, we find elements in $\THC^\ast_S(\ell)$ that pair nicely with the generators we found in early sections. In Section~\ref{sec:SSB4}, we use the vanishing and cyclicity results we found in Section~\ref{sec:SSB3} to find all of the differentials and extensions in the fourth Bockstein spectral sequence. This completes the computation of $\THH(\ell)$.

We round out our computations in Section~\ref{sec:ko}, where we calculate the $2$-local homotopy of $\THH(ko)$. Finding $\THH_\ast(ko;ku)$ requires an analysis similar to that for $\THH_\ast(ku)$, and we pass from $\THH_\ast(ko;ku)$ to $\THH_*(ko)$ by analyzing the $\eta$-Bockstein spectral sequence and resolving hidden extensions.

\section{Preliminary Remarks and Statement of Results}\label{sec:prelim}

\subsection{Algebraic Preliminaries}\label{sec:algprelim}
As a global piece of notation, we will write $a \doteq b$ when $a$ is equal to $b$ up to multiplication by a $p$-local unit.

We begin with a few lemmas which allow us to state the kinds of B\"okstedt spectral sequences we will use. If $R$ is an $S$-algebra and $M$ is an $R$-bimodule, let $\THH(R;M)$ denote the derived smash product $M \sma_{R \sma R^{op}} R$ \cite[\S~IX]{EKMM}. If instead of $S$-algebras we consider $E$-algebras for a fixed $E_\infty$ ring spectrum $E$, then we will denote the relative $\THH$ by $\THH^E(R;M)$.

\begin{lemma} \label{lem:THHbimod}
Suppose $R$ is a commutative $S$-algebra and $M$ is an $R$-module given the commutative bimodule structure.  Then there is a weak equivalence
\[
\THH(R;M) \simeq M \sma_R \THH(R).
\]
\end{lemma}

\begin{proof}
We show a chain of weak equivalences whose composite is the one in question. By definition, we have
\[
\THH(R;M)\simeq M\sma_{R \sma R^{op}}R\simeq (M \sma_R R)\sma_{R\sma R^{op}} R,
\]
and reassociating gives that this is equivalent to
\[
M\sma_R(R \sma_{R \sma R^{op}} R)\simeq M \sma_R\THH(R).\qedhere
\]
\end{proof}

\begin{lemma}
Suppose $R \to Q$ is a map of $S$-algebras and $M$ is a $Q$-$R$ bimodule, given an $R$-$R$ bimodule structure by pullback.  Then there is a weak equivalence
\[
\THH(R;M) \simeq M \sma_{Q \sma R^{op}} Q.
\]
\end{lemma}

\begin{proof}
Similarly to the previous lemma, this follows from the following chain of weak equivalences:
\[
\THH(R;M) \simeq M \sma_{R \sma R^{op}} R \simeq (M \sma_{Q \sma R^{op}} Q
\sma R^{op}) \sma_{R \sma R^{op}} R,
\]
and reassociating gives that this is equivalent to
\[
M \sma_{Q \sma R^{op}} (Q \sma R^{op} \sma_{R \sma R^{op}} R) \simeq M
\sma_{Q \sma R^{op}} Q.\qedhere
\]
\end{proof}

This gives a K\"unneth spectral sequence computing the homotopy of this derived smash product \cite[IV 4.1]{EKMM}.

\begin{cor} \label{cor:KunnethSS}
Under these circumstances, there is a K\"unneth spectral sequence with $E_2$-term
\[
\Tor_{\ast\ast}^{Q_\ast R^{op}} (M_\ast, Q_\ast) \Rightarrow \THH_\ast(R;M).
\]
\end{cor}

This expression for topological Hochschild homology often leads to strictly simpler computations than are usually carried out by means of the B\"okstedt spectral sequence \cite{Bo1, Bo2}.  For instance, if $R=Q=H\F_p$, we obtain a spectral sequence starting from
\[
\Tor^{A_\ast}_{\ast\ast}(\F_p, \F_p) \Rightarrow \THH_\ast(\F_p).
\]
Here $A_\ast$ is the dual Steenrod algebra.  This can be identified as the part of the B\"okstedt spectral sequence consisting of the primitives under the $A_\ast$-comodule action.

There are dual statements for topological Hochschild cohomology that we will need in Section~\ref{sec:THCku}. Let $\THC_{E}(R;M)$ denote the derived function spectrum $F_{R \sma_{E} R^{op}}(R,M)$. Even if $E=S$, the sphere spectrum, we will include it in the notation to distinguish from the topological Hochschild homology spectrum.

\begin{lemma}[\cite{EKMM}]
\label{lem:dualize}
Suppose $R \to Q$ is a map of $E$-algebras and $M$ is a $Q$-$R$ bimodule, given an $R$-$R$ bimodule structure by pullback.  Then there is a weak equivalence
\[
\THC_E(R;M) \simeq F_{Q \sma_E R^{op}} (Q,M).
\]
\end{lemma}

This in turn gives a universal coefficients spectral sequence which we will use to compute $\THH_{E}^\ast(R;M)$ \cite[IV 4.1]{EKMM}.

\begin{cor}\label{cor:UnivCoefs}
Under these circumstances, there is a universal coefficient spectral sequence
\[
\Ext^{\ast\ast}_{\pi_\ast(Q\sma_E R^{op})} (Q_\ast, M_\ast) \Rightarrow \THC^{\ast}_E(R;M).
\]
\end{cor}

\subsection{Method and Main Results}
Recall that as an algebra,
\[
\pi_\ast\ell=\Z_{(p)}[v_1],
\]
where $|v_1|=2p-2$. Our key technique is to play the reductions modulo $p$ and $v_1$ off of each other in computable ways.

Computations using the B\"okstedt spectral sequence or Corollary~\ref{cor:KunnethSS} allow us to see that
\[
\THH_\ast(\ell;H\F_p)=E(\lambda_1,\lambda_2)\otimes \F_p[\mu],
\]
where $|\lambda_i|=2p^{i}-1$, and $|\mu|=2p^2$ \cite{AnRo,McSt}. Moreover, since the bimodule $H\F_p$ is the quotient of $\ell$ by $p=v_0$ and $v_1$, we can find two intermediate $\ell$-modules between $\ell$ and $H\F_p$, namely Morava $k(1)=\ell/p$ and $H\Z_{(p)}=\ell/v_1$. This allows us to go from $\THH(\ell)$ to $\THH(\ell;H\F_p)$ in two ways:
\begin{equation*}
\xymatrix{{} & {\THH(\ell)} \ar[rd] \ar[ld] & {} \\ {\THH(\ell;H\Z_{(p)})}
\ar[rd] & {} & {\THH(\ell;k(1))} \ar[ld] & {}
\\ {} & {\THH(\ell;H\F_p)} & {}}
\end{equation*}

Each of the arrows in the above diagram gives a Miller-Novikov Bockstein spectral sequence going the other way \cite{Ra86}.
The construction of each is identical: for $M$ one of these bimodules and for $x$ either $p$ or $v_1$, we have a cofiber sequence of $\ell$-bimodules
\[
\Sigma^{|x|}M\xrightarrow{x} M\to M/x.
\]
Iterating this gives us a filtration of $M$, the filtration quotients of which are the iterated $|x|$-fold suspensions of $M/x$.
Thus we have the following spectral sequences:
\begin{eqnarray}
\THH_\ast(\ell;H \F_p)[v_1] & \Longrightarrow &
\THH_\ast\!\big(\ell;k(1)\big); \label{B1} \\
\THH_\ast(\ell;H \F_p)[v_0] & \Longrightarrow &
\THH_\ast(\ell;H\Z_{(p)})^\wedge_p; \label{B2} \\
\THH_\ast\!\big(\ell;k(1)\big)[v_0] & \Longrightarrow &
\THH_\ast(\ell)^\wedge_p; \label{B3} \\
\THH_\ast(\ell;H\Z_{(p)})[v_1] & \Longrightarrow & \THH_\ast(\ell).
\label{B4}
\end{eqnarray}
These are bigraded spectral sequences in which elements of $\THH_\ast(\ell;M)$ have bidegree $(\ast,0)$ and in which $v_0$ or $v_1$ have bidegrees $(0,1)$ or $(2p-2,1)$ respectively. These spectral sequences have Adams style differentials, and in all cases except the third, each of these spectral sequences can be realized as an Adams spectral sequence in an appropriate category of module spectra (to ensure equality on $E_1$, we choose a minimal resolution of our module).

We can understand the first two spectral sequences easily, and this gives us two spectral sequences which we can play against each other to understand $\THH_\ast(\ell)$. Moreover, since $\THH(\ell)$ is finitely generated, these results actually give $p$-local, rather than $p$-complete, information.

We recall that for a commutative $S$-algebra $R$ with a module $M$ given the commutative bimodule structure, there is a splitting in $R$-modules
\[
\THH(R;M) \simeq M \vee \rTHH(R;M).
\]
For convenience, we will often perform computations with $\rTHH(\ell)$ and exclude the factor of $\ell$ which splits off.

We will show how to completely understand $\THH_\ast(\ell)$ as an $\ell_\ast$-module.

\begin{thm}
As an $\ell_\ast$-module,
\[
\THH_\ast(\ell)=\ell_\ast\oplus \Sigma^{2p-1} F\oplus T,
\]
where $F$ is a torsion free summand and $T$ is an infinite direct sum of torsion modules concentrated in even degrees.
\end{thm}

\subsection{The Torsion Free Part}
Since rational homotopy is rational homology, we can easily run the B\"okstedt spectral sequence computing the rational homotopy of $\THH(\ell)$. For degree reasons, the spectral sequence collapses with no possible extensions, and we have as isomorphism of $\ell_\ast\otimes\Q$-algebras
\[
\THH_\ast(\ell)\otimes\Q\cong\Q[\lambda_1,v_1]/\lambda_1^2,
\]
where $|\lambda_1|=2p-1$. This tells us exactly where all of the torsion free summands of $\THH_\ast(\ell)$ lie.

\begin{theorem}\label{thm:v1Divisibility}
The torsion free summand of $\rTHH_\ast(\ell)$ is $F\cdot\lambda_1$, where $F$ the $\ell_\ast$-module
\[
F = \ell_*\left[\frac{v_1^{p^k+\dots+p}}{p^k};\, k\geq 1\right] \subset \ell_*\otimes\Q.
\]
\end{theorem}

Thus the classes $v_1^k \lambda_1$ become increasingly $p$-divisible as $k$ gets large.
We pause to mention the relation of this structure to known computations. At an odd prime McClure and Staffeldt already found (\cite[Theorem 8.1]{McSt}) that
\[
\THH(L) \simeq v_1^{-1} \THH(\ell)\simeq L \vee \Sigma L_{\Q},
\]
where $L$ is the periodic Adams summand. This can be seen directly here at any prime, as follows. Inverting $v_1$ in the homotopy of $\THH(\ell)$ leaves $L\sma_{\ell}\THH(\ell)=\THH(\ell;L)$. However, $L$ is the localization of the $\ell\sma{}\ell^{op}$-module $\ell$ obtained by inverting both images of $v_1$. Localization of modules (in the sense of inverting elements in homotopy) over a commutative ring spectrum commutes with taking smash product with other modules, and so we find that inverting $v_1$ yields
\[
\THH(\ell;L)\simeq L\sma_{L\sma{}L^{op}}L=\THH(L).
\]
Therefore, inverting $v_1$ in the homotopy of $\THH(\ell)$ yields
\[
\THH_\ast(L)=L_\ast\oplus v_1^{-1}\Sigma^{2p-1}F\cong L_\ast\oplus\Sigma(L_\Q)_\ast,
\]
recovering the McClure-Staffeldt result for odd primes and extending it to $p=2$.

\subsection{The Torsion Part}
The torsion is rather involved, but it can also be understood. It is concentrated in even degrees, and it follows a kind of tower-of-Hanoi pattern with increasingly complicated, inductively built, components.

We define a sequence of torsion modules $T_n$ for $n \geq 0$ as follows.  As an $\ell_*$-module, each $T_n$ has generators $g_w$ for all strings $w$ on letters $0,\dots, p-1$. We impose two kinds of relations. First, we require that $g_w = 0$ if $|w|>n$, where $|w|$ denotes the length of the string. Second, if we write $w \cdot w'$ for the concatenation of strings, we have
\[
p g_w =
\begin{cases}
v_1^{p^{(n-|w|+2)}}g_{w'} + g_{w \cdot 0} &\text{if } w = w' \cdot (p-1),\\
g_{w \cdot 0} &\text{otherwise.}
\end{cases}
\]
One can show inductively that that these relations imply that
\[
v_1^{p^{n-|w|+1}+\dots+p} g_w = 0
\]
for all $w$. The non-zero generators are graded by saying that if $w=a_1\dots a_k$, then
\[
|g_w|=2p^2(a_1p^{n-1}+\dots+a_kp^{n-k}).
\]

An easy example is that $T_0 = \ell_*/(p,v_1^p)$.

More generally, the modules $T_n$ have only finitely many nonzero elements.  The modules $T_n$ are self-dual; the duality is given by
\[
g_w \longleftrightarrow v_1^{p^{n-|w|+1}+\dots+p-1} g_{\bar w},
\]
where $\bar w$ is the string $w$ with each digit $a$ replaced by $p-1-a$.

The modules $T_n$ are more easily viewed through a recursive construction. First there are inclusions of direct summands $\Sigma^{2kp^{n+1}}T_{n-1}\hookrightarrow T_n$ given by $g_w\mapsto g_{k\cdot w}$, for $1\leq k\leq p-2$. There is also an inclusion of $T_{n-1}$ given by $g_w\mapsto g_{0\cdot w}$ and a projection $T_n\to \Sigma^{2(p-1)p^{n+1}}T_{n-1}$ given by
\[
g_w\mapsto\begin{cases}
g_{w'} & w=(p-1)\cdot w' \\
0 & \text{otherwise.}\end{cases}
\]
The structure of $T_n$ as a module is determined by having these summands, submodules, and quotient modules, together with a generator $g_{\varnothing}$ satisfying relations
\begin{align*}
p\cdot g_\varnothing &=g_0\\
v_1^{p^{n+1}}g_\varnothing &=g_{(p-1)0}-p\cdot g_{(p-1)}.
\end{align*}

One can view $T_n$ as being recursively constructed out of $p$ copies of $T_{n-1}$, where we glue together the first and last copies along a $v_1$-tower of length $p^{n+1}+\dots+p$.

\begin{theorem}\label{thm:Torsion}
The torsion summand of the homotopy of $\THH(\ell)$, as an $\ell_*$-module, is isomorphic to
\[
\bigoplus_{n\geq 0}\bigoplus_{k=1}^{p-1}
\Sigma^{2kp^{n+2} + 2(p-1)}T_n.
\]
In particular, for all $n\geq 1$ and $2\leq k\leq p$, the even dimensional homotopy between degrees $2kp^{n+2}-2p+1$ and $2kp^{n+2}+2p-3$ is zero.
\end{theorem}

To facilitate understanding of the modules $T_n$, we have included a picture of the torsion for $p=2$ starting in degrees $18$ and $34$ as Figure~\ref{fig:Torsion}. These correspond to $T_1$ and $T_2$.
\begin{center}
\begin{figure}[ht]
\includegraphics[width=.9\textwidth]{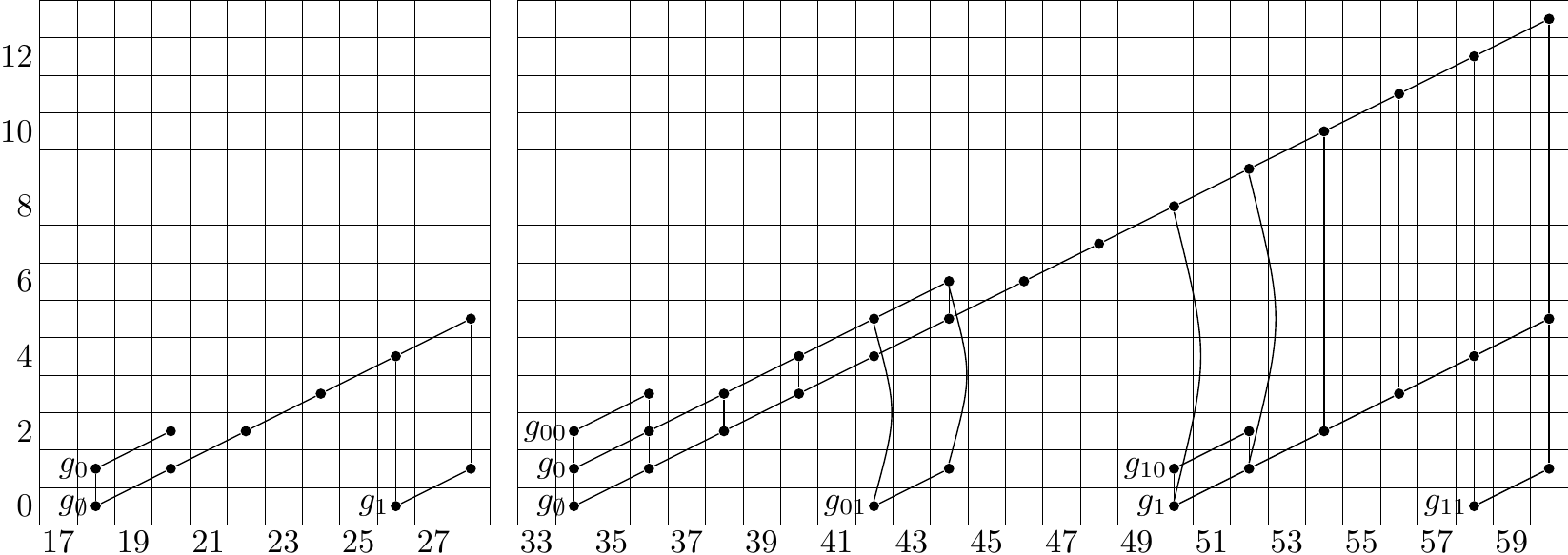}
\caption{The torsion in degrees $18$ through $60$ for $p=2$}
\label{fig:Torsion}
\end{figure}
\end{center}

\section{The first two Bockstein spectral sequences}\label{sec:BSS1and2}
In this section, we compute the base case Bockstein spectral sequences, Spectral Sequences (\ref{B1}) and (\ref{B2}). As was mentioned above,
\[
\THH_\ast(\ell;H\F_p)=E(\lambda_1,\lambda_2)\otimes\F_p[\mu],
\]
where $\lambda_1$, $\lambda_2$ and $\mu$ are in degrees $2p-1$, $2p^2-1$ and $2p^2$, respectively. Here $\lambda_1$ is represented by $\sigma \xi_1$, $\lambda_2$ is
represented by $\sigma \xi_2$ and $\mu$ is represented by $\sigma \tau_2$. The elements $\xi_i$ and $\tau_2$ arise from the dual Steenrod algebra via the change of rings, and the operator $\sigma$ represents multiplication by the fundamental class of $S^1$,
using circle action on $\THH(\ell)$. At $p=2$, the usual modifications involving the names of classes in the dual Steenrod algebra apply. We note, as in \cite[Prop 4.2]{McSt} or \cite[Thm 5.12]{AnRo}, that there is a mod $p$ Bockstein connecting $\mu$ and
$\lambda_2$. This is essential for starting the $H\Z$-Bockstein spectral sequence.

\subsection{The $k(1)$-Bockstein spectral sequence}
McClure and Staffeldt ran the first Bockstein spectral sequence at an odd prime, calculating the homotopy groups of $\THH(\ell;k(1))$ in \cite{McSt}, and in \cite{AnRo}, Rognes and the first author extended the calculation to $p=2$.

The calculation depends on the following result of McClure and Staffeldt:
\[
HH_\ast(K(1)_\ast\ell)\cong K(1)_\ast\ell.
\]
This implies that the $K(1)_*$-based B\"okstedt spectral sequence collapses: $K(1)_*\ell \cong K(1)_*\!\THH(\ell)$, and with the exception of the class $1$, all elements are $v_1$-torsion. There is only one pattern of differentials compatible with this, and it produces $v_1$-towers of various length on $\mu^i \lambda_1$, $\mu^i\lambda_1\lambda_2$ and $\mu^{pi} \lambda_2$. For the reader's convenience we recall the result here.

Recursively define $r(n)$ by $r(1)=p$, $r(2)=p^2$ and $r(n)=p^n+r(n-2)$ for $n \geq 3$. Also define $\lambda_n$ by $\lambda_n=\lambda_{n-2}\mu^{p^{n-3}(p-1)}$.

\begin{theorem}[\cite{AnRo, McSt}]\label{thm:k1SS}
The homotopy of $\THH(\ell;k(1))$ is generated as a module over $\F_p[v_1]$ by $1$, $x_{n,m}=\lambda_n \mu^{p^{n-1} m}$ and $x'_{n,m}=\lambda_n \lambda_{n+1} \mu^{p^{n-1} m}$ for $n \geq 1$ and $m \geq 0$, $m\not \equiv p-1\mod p$. The relations are generated by
\[
v_1^{r(n)} x_{n,m}=v_1^{r(n)} x'_{n,m}=0.
\]
\end{theorem}

Figure~\ref{fig:V1} shows the homotopy of $\rTHH\big(ku;k(1)\big)$ through dimension $33$ at $p=2$.

\begin{center}
\begin{figure}[ht]
\includegraphics[width=.9\textwidth]{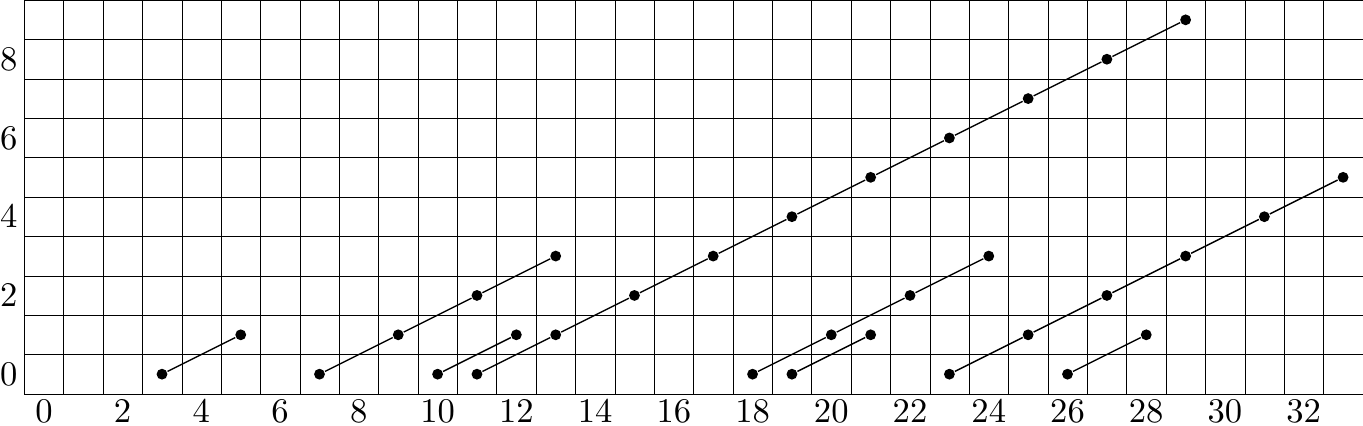}
\caption{The homotopy of $\rTHH\big(ku;k(1)\big)$} \label{fig:V1}
\end{figure}
\end{center}

\subsection{The $H\Z$-Bockstein spectral sequence}\label{sec:HZSS}
Next we run the Bockstein spectral sequence converging to
$\THH_\ast(\ell; H\Z_{(p)})^\wedge_p$. There is an immediate differential $d_1(\mu)=v_0\lambda_2$, since as was described above, the corresponding classes in the homology of $\THH(\ell)$ are connected by a Bockstein.

To get the remaining differentials, we use the following ``Leibniz'' rule for higher differentials in a Bockstein spectral sequence.

\begin{lemma}
We have differentials
\[
d_{i+1}(\mu^{p^i})=v_0^{i+1} \mu^{p^i-1} \lambda_2.
\]
\end{lemma}

\begin{proof}
We use a result of May relating the higher Bocksteins and $p^{\text{th}}$ powers \cite[Proposition 6.8]{Ma70} in an $E_\infty$ context. If $x$ supports a $d_j$ differential in the Bockstein spectral sequence, then
\[
d_{j+1}(x^p)=v_0 x^{p-1} d_j(x)
\]
if $p>2$ or if $p=2$, $j \geq 2$. If $p=2$ and $j=1$, then there is an error term of $\mathcal P_4(d_1(x))$.

In our case, it is easy to see that the error term vanishes. The error term is $\mathcal P_{4}(\lambda_2)$. This power operation is simply $Q^8$, and since $Q^8$ commutes with the $\sigma$ action, we see that
\[
Q^8(\lambda_2)=Q^8(\sigma\xi_2^2)=\sigma Q^8(\xi_2^2)=\sigma (Q^4\xi_2)^2=\sigma\xi_3^2=0.\qedhere
\]
\end{proof}

\begin{remark}
Just as we used the $K(1)$-based B\"okstedt spectral sequence to get $v_1$-torsion information, we can use the $K(0)=H\Q$-based B\"okstedt spectral sequence to get $p$-torsion information. The $H\Q$-based B\"okstedt spectral sequence gives
\[
\THH_\ast(\ell;H\Q) \cong \Q[\lambda_1]/\lambda_1^2,
\]
so we know that all the rest is torsion. This method alone tells us where all the differentials are, though not how long they are in this case.
\end{remark}

Let $a_i=\mu^{i-1} \lambda_2$ and $b_i=\mu^{i-1} \lambda_1 \lambda_2$, for $i\geq 1$. Then $|a_i|=2p^2i-1$ and $|b_i|=2p^2i+2(p-1)$, and they both have order $p^{k+1}$, where $k=\nu_p(i)$, the $p$-adic valuation of $i$. The above analysis shows the following proposition.

\begin{prop}
The homotopy of $\rTHH(\ell;H\Z_{(p)})$ is a copy of $\Z_{(p)}$ generated by $\lambda_1$ plus torsion. The torsion is generated as a $\Z_{(p)}$-module by the elements $a_i$ and $b_i$.
\end{prop}

Since this is a Bockstein spectral sequence for replacing $p$, we know that there are no possible additive extensions. The lifts of $a_i$ and $b_i$ to $\rTHH_\ast(\ell; H\Z_{(p)})$ are defined up to a $p$-adic unit.

\section{The third Bockstein spectral sequence}\label{sec:SSB3}
In this section, we say as much as we can about the spectral sequence
\[
\rTHH_\ast\!\big(\ell;k(1)\big)[v_0]\Rightarrow\rTHH_\ast(\ell)_p^{\wedge}.
\]
The $E_1$-page through dimension $33$ for $p=2$ is depicted in Figure~\ref{fig:BSS1}. Note that $v_1$ is really in filtration $0$ here; we draw it in filtration $1$ to reduce the clutter.

\begin{center}
\begin{figure}[ht]
\includegraphics[width=.9\textwidth]{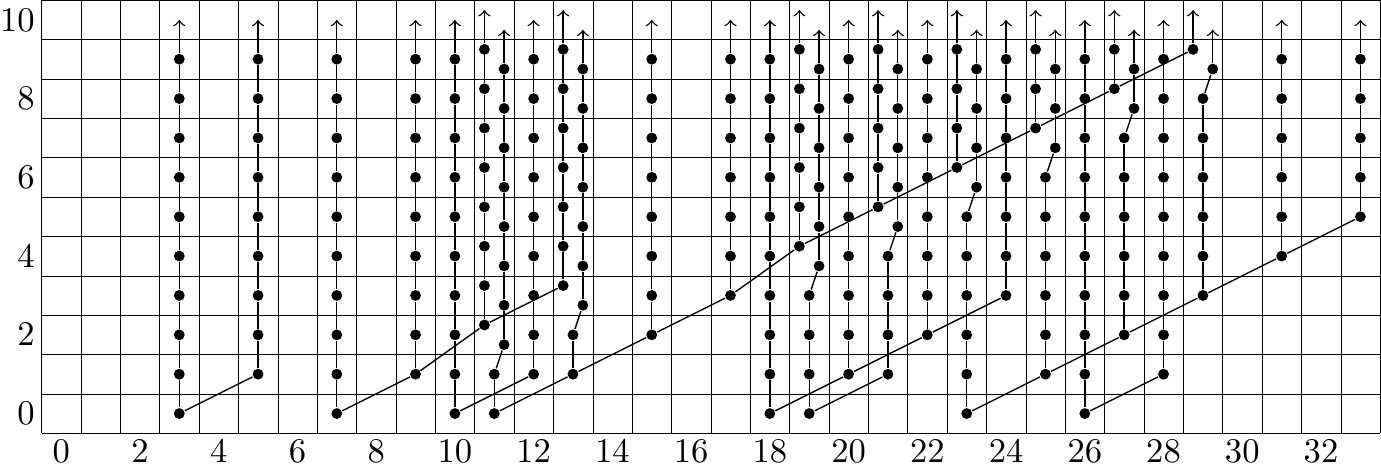}
\caption{The $v_0$-Bockstein spectral sequence converging to $\rTHH_\ast(ku)$} \label{fig:BSS1}
\end{figure}
\end{center}

This spectral sequence allows us to conclude key facts about some of the homotopy groups, making upper bounds on the order of the groups or determining if they are cyclic. The following lemmata serve as the workhorses for all computations in Spectral Sequence (\ref{B4}).

\begin{lemma}\label{lem:zerocyclic}
\mbox{}
\begin{enumerate}
\item[(a)] The groups $\rTHH_{2p^{n+2}-2p}(\ell)$ for $n \geq 1$ are cyclic.
\item[(b)] The groups $\rTHH_{2p^{n+2}-2}(\ell)$ and $\rTHH_{2p^{n+2}}(\ell)$ for $n \geq 0$ are 0.
\item[(c)] The groups $\rTHH_{2p^{n+2}+2p-2}(\ell)$ for $n \geq 0$ are cyclic.
\end{enumerate}
\end{lemma}

\begin{proof}
We show this by showing that in this Bockstein spectral sequence, the $E_1$ term is $\F_p[v_0]$ in degree $2p^{n+2}-2p$ and $2p^{n+2}+2p-2$ and is $0$ in degrees $2p^{n+2}-2$ and $2p^{n+2}$.

The only even generators from Theorem~\ref{thm:k1SS} are the generators $x_{n,m}'$. A simple counting argument shows that
\[
|x_{n,m}'|=2(p^{n+1}-1)+2p+2mp^{n+1}=2(p-1)+2(m+1)p^{n+1}.
\]
The element $x_{n,m}'$ supports a $v_1$-tower truncated at height $r(n)$, where $r(n)$ was recursively defined for Theorem~\ref{thm:k1SS} as well. We can now check by degree, arguing by $p$-adic expansion.

The proof of the results in the lemma are very similar. We must first find all of the elements of the form $x_{j,m}'v_1^k$ in dimension $2p^{n+2}-2p$. In other words, we must find all triples $(j,m,k)$ such that
\[
2p^{n+2}-2p=2(m+1)p^{j+1}+(1+k)2(p-1),
\]
subject to the condition that $k<r(j)$. The integer $2(p^{n+2}-p)$ is divisible by $2(p-1)$, so we see that $m+1=(p-1)\hat{m}$ for some integer $\hat{m}$. Dividing through by $2(p-1)$ leaves
\[
\hat{m}p^{j+1}+1+k=p^{n+1}+\dots+p.
\]
In particular, we must have that
\[
k\geq p^{j}+\dots+p-1.
\]
If $j>1$, then this quantity is bigger than $r(j)$ by induction. If $j=1$, then there is a solution with $k<r(j)$, namely $k=p-1$, $\hat{m}=1+\dots+p^{n-1}$. This means that modulo $p$, there is one generator in degree $2p^{n+2}-2p$, namely $x_{1,p^n-2}'v_1^{p-1}$.

For degree $2p^{n+2}-2$, we note that this degree is the degree just argued plus $|v_1|$. In that degree, the generating class was the largest allowed $v_1$ multiple of the class $x_{1,p^n-2}'$, so there can be no classes in $2p^{n+2}-2$.

For degree $2p^{n+2}$, we search for classes $x_{j,m}'v_1^k$ such that
\[
|x_{j,m}'v_1^k|=2(m+1)p^{j+1}+2(p-1)+2(p-1)k=2p^{n+2}.
\]
Combining terms and reducing modulo $p^{j+1}$, we see that $k$ is at least $p^{j+1}-1$, which is in particular larger than $r(j)$.

The proof of part (c) is similar. Part (b) shows that there are no $v_1$-divisible classes in this degree. By construction, there is at most one $x_{j,m}'$ in each degree, and in this degree, we have the classes $x_{n+1,0}'$.
\end{proof}

\begin{lemma}\label{lem:ExtensionLemma}
The groups $\rTHH_{2p^{n+2}-1}(\ell)$ for $n \geq 0$ are $\Z_{(p)}$.
\end{lemma}

\begin{proof}
The proof of this lemma is very similar. Here we consider the odd classes $x_{j,m}$, and the argument depends on the parity of $j$. A combinatorial check shows that
\[
|x_{j,m}|=\begin{cases}
             2(p-1)\big(p^{j-1}+p^{j-3}+\dots+p\big)+2p-1+2mp^{n+1} &
             j\text{ even} \\
         2(p-1)\big(p^{j-1}+p^{j-3}+\dots+1\big)+1+2mp^{n+1} &
         j\text{ odd}.
            \end{cases}
\]
We must find those values of $j$, $m$, and $k$ such that
\[
|x_{j,m}v_1^k|=2p^{n+2}-1.
\]
Regardless of the parity of $j$, if we subtract $1$ from both sides, then the right hand side is $0$ modulo $2(p-1)$. This implies that again $m=\hat{m}(p-1)$. At this point, the argument does not depend on the parity in an essential way, so we spell out only the case of $j=2i$. Dividing by $2(p-1)$ leaves
\[
\hat{m}p^{2i+1}+p^{2i-1}+\dots+p^3+p+1+k=p^{n+1}+\dots+1.
\]
If $n+1>2i-1$, then $k\geq p^{2i}+\dots+p^2=r(2i)$. On the other hand, if $n+1=2i-1$, then we can choose $k=r(2i-2)$, $\hat{m}=0$. Thus we have for each $n$ a single generator in degree $2p^{n+2}-1$, namely $x_{n+2,0}v_1^{r(n)}$.
This class must generate an infinite cyclic group by the computation of $THH_*(\ell) \otimes \Q$.
\end{proof}

\begin{lemma}\label{lem:bpiCycles}
The groups $\rTHH_{2p^{n+2}+2p-3}(\ell)$ for $n \geq 0$ are $\Z_{(p)}$.
\end{lemma}

\begin{proof}
The proof is the same as for the previous case. The dimension in question is that of the previous lemma plus $|v_1|$. Here we must choose $k=r(n)+1$, and the generating class is $x_{n+2,0}v_1^{r(n)+1}$.
\end{proof}

\begin{cor}\label{cor:xn0primesurvive}
For all $i$, the classes $x_{i+1,0}'$ survive the Bockstein spectral sequence, giving non-zero permanent cycles in $\rTHH_\ast(\ell)$.
\end{cor}

\begin{proof}
Part (c) of Lemma~\ref{lem:zerocyclic} shows that $\rTHH_{2p^{n+2}+2p-2}(\ell)$ is a cyclic group generated by $x_{i+1,0}'$, and Lemma~\ref{lem:bpiCycles} shows that $\rTHH_{2p^{n+2}+2p-3}(\ell) \cong \Z_{(p)}$. Hence $x_{i+1,0}'$ cannot support a differential.
\end{proof}

\section{Topological Hochschild cohomology}\label{sec:THCku}
To finish the computations, we will use the ``cap product'' pairing of topological Hochschild cohomology with topological Hochschild homology:
\[
\THC^n_S(\ell)\otimes\THH_m(\ell)\to \THH_{m-n}(\ell).
\]
This arises quite naturally. The spectrum $THH_S(\ell)$ is the function spectrum $F_{\ell\sma \ell^{op}}(\ell,\ell)$ of $(\ell-\ell)$-bimodule maps from $\ell$ to itself. The smash product in $(\ell\wedge\ell^{op})$-modules is functorial in each factor, so there is a canonical map
\[
THH_S(\ell)\wedge THH^S(\ell)\to THH^S(\ell)
\]
given by ``evaluating the function on the first factor''. Our cap product is the effect in homotopy of this pairing.

Many of the torsion patterns, both for the $v_0$ and $v_1$ Bockstein spectral sequences, arise from multiplication by powers of $\mu$. While no power of $\mu$ survives the Bockstein spectral sequences, the $\mu^k$ translates of permanent cycles do survive. Using the pairing with $\THC_S(\ell)$, we can actually connect these elements on the $E_\infty$-page.

We first note that certain topological Hochschild cohomology spectra inherit ``Hopf algebra'' type structures.  It was proven in \cite{AnRo} that when $R$ is commutative, $\THH(R)$ has the structure of a Hopf algebra spectrum over $R$, and hence the base extension $\THH(R;Q) \cong \THH(R) \sma_R Q$ inherits a Hopf algebra spectrum structure over $Q$ when $Q$ is a commutative $R$-algebra.

Let $D_Q$ denote the $Q$-Spanier-Whitehead dualization functor.  The dual $D_Q(C)$ of a Hopf algebra spectrum $C$ over $Q$ inherits a $Q$-algebra structure from the coalgebra structure, and there are natural maps of $Q$-algebras
\[
D_Q(C) \to D_Q(C \sma_Q C) \leftarrow D_Q(C) \sma_Q D_Q(C),
\]
where the first map is induced by the multiplication.  If the second map is a weak equivalence (such as when $C$ is a finite cell object, or when $Q$ is $H\Z_{(p)}$ or $H\F_p$ and $C$ has finitely generated homotopy groups), there is an induced Hopf algebra spectrum structure on $D_Q(C)$ up to homotopy.  In particular, $\THC_S(\ell;H\F_p)$ and $\THC_S(\ell;H\Z_{(p)})$ are both Hopf algebra spectra over $H\F_p$ and $H\Z_{(p)}$ respectively.

Since $\THH(\ell;H\F_p)$ is an $H\F_p$-algebra with finitely generated homotopy groups, we can dualize its homotopy groups directly to conclude that as a Hopf algebra,
\[
\THC^\ast_S(\ell;H\F_p)=E(x_{2p-1},x_{2p^2-1})\otimes \Gamma(c_1),
\]
where $\Gamma(c_1)$ denotes a divided power algebra on a class $c_1$ in degree $2p^2$, and the generators $x_{2p-1}$, $x_{2p^2-1}$, and $c_1$ are primitive.  The divided power generator $c_k=\gamma_k(c_1)$ is dual to $\mu^k$, so if it survives the $H\F_p$-based Adams spectral sequence in the category of $\ell$-modules to give an element of $\THC^\ast_S(\ell)$, then capping with it will undo the multiplications by $\mu^k$ that were seen on the $E_\infty$-page. However, since this module over the Steenrod algebra is negatively graded and not bounded below, there are convergence problems with the Adams spectral sequence. We instead compare with relative topological Hochschild cohomology.


\subsection{Relative topological Hochschild cohomology of $\ell$}
We write the remainder of the section with the assumption that $BP$ is an $E_\infty$ ring spectrum. If this is not the case, then we can replace $BP$ with $MU$. Many of the key points are the same; the notation is slightly simpler in the $BP$ case. To streamline notation further, we also let $\tau_i=\xi_{i+1}$ if $p=2$. We begin by recalling the homology of $\ell$ and $BP$. As is standard, we denote the image of a class under the canonical anti-automorphism by an over-line.

\begin{prop}
As an $\A_\ast$-sub-comodule algebra of $\A_\ast$,
\[
H_\ast\ell=\begin{cases}
\F_2[\bar{\tau}_0^2,\bar{\tau}_1^2,\bar{\tau}_2,\dots] & p=2, \\
\F_p[\bar{\xi}_1,\dots]\otimes E(\bar{\tau}_2,\dots) & p>2,
\end{cases}
\]
and
\[
H_\ast BP=\begin{cases}
\F_2[\bar{\tau}_0^2,\dots] & p=2, \\
\F_p[\bar{\xi}_1,\dots] & p>2.
\end{cases}
\]
\end{prop}

\begin{prop}
As a ring,
\[
\pi_\ast H\F_p\wedge_{BP}\ell=E(\bar{\tau}_2,\bar{\tau}_3 \dots),
\]
and the map from ${H\F_p}_* \ell$ induced by the unit $S^0\to BP$ is the canonical quotient.
\end{prop}
\begin{proof}
We use the equivalence in ring spectra
\[
H\F_p\wedge_{BP}\ell\simeq H\F_p\wedge_{H\F_p\wedge BP}(H\F_p\wedge\ell).
\]
The K\"unneth theorem then gives both parts of the theorem, since $H_*(\ell;\F_p)$ is free over $H_*(BP;\F_p)$.
\end{proof}

The universal coefficient spectral sequence on the above exterior algebra then collapses, telling us that
\[
\THC_{BP}^\ast(\ell;H\F_p)=\F_p[e_1,e_2,\dots],
\]
where $e_i$ is the class in $\Ext$ corresponding to $\bar{\tau}_{i+1}$.

Since this is concentrated in even degrees, we conclude that the Bockstein spectral sequences taking us from $\THC_{BP}(\ell;H\F_p)$ to $\THC_{BP}(\ell)$ collapse, giving
\[
\THC_{BP}^\ast(\ell)=\ell_\ast[\![e_1,\dots]\!].
\]

The structure map $S^0\to BP$ induces a commutative diagram
\[
\xymatrix{{\THC_{BP}(\ell)}\ar[d] \ar[r] & {\THC_{BP}(\ell;H\F_p)} \ar[d]
\\
{\THC_S(\ell)} \ar[r] & {\THC_S(\ell;H\F_p)}}
\]

We want to show that the elements $c_k$ in $\THC^\ast(\ell;H\F_p)$ lift to $\THC^\ast(\ell)$. However, we can see this using the commutativity of the above diagram.

\begin{prop}
The map from $\THC^\ast_{BP}(\ell;H\F_p)$ to $\THC^\ast_S(\ell;H\F_p)$ sends $e_k$ to $c_{p^{k-1}}$.
\end{prop}

\begin{proof}
This is immediate from our discussion of the map in homotopy
\[
\pi_\ast(H\F_p\sma \ell)\to \pi_\ast(H\F_p\sma_{BP}\ell)
\]
induced by the unit $S^0\to BP$. The classical B\"okstedt spectral sequence identifies the generators in $\Ext_{H_*(\ell)}$ coming from $\bar{\tau}_{k+2}$ with $\gamma_{p^{k}}(c_1)$.
\end{proof}

\begin{remark}
In order to use $\THC$ relative to $MU$ rather than relative to $BP$, the following changes must be noted.  The ring $\pi_*(H\F_p \wedge_{MU} \ell)$ is the ring previously calculated as $\pi_*(H\F_p \wedge_{BP} \ell)$ tensored with an exterior algebra on classes in odd degrees. The universal coefficient spectral sequence then shows that
$\THC_{MU}^*(\ell;H\F_p)$ is the tensor product of the ring calculated as $\THC_{BP}^*(\ell;H\F_p)$ with a polynomial algebra on generators in even degrees.
\end{remark}

We can therefore conclude that in fact the elements $c_k$ all survive to homotopy classes in $\THC^\ast_S(\ell)$. It should be noted, however, that this method does not rule out the possibility that they are torsion classes. To better understand this, we analyze $\THC^\ast_S(\ell;H\Z_{(p)}).$

\begin{thm}\label{thm:THCHopf}
As a Hopf algebra,
\[
\THC^\ast_S(\ell;H\Z_{(p)})=E(x_{2p-1}) \otimes \Gamma(c_1)/(pc_1),
\]
where $x_{2p-1}$ and $c_1$ are again primitive.
\end{thm}

\begin{remark}
As $\THC_S^*(\ell;H\Z_{(p)})$ is not flat over $\Z_{(p)}$, it is not immediate that the comultiplication on the topological Hochschild cohomology spectrum gives rise to a comultiplication on the level of homotopy groups.  However, the classes in $\THC_S^*(\ell;H\Z_{(p)})$ lie in degrees congruent to $0$ and $2p-1$ mod $2p^2$, and hence the Tor-terms in the homotopy groups of
\[ \THC_S(\ell;H\Z_{(p)}) \sma_{H\Z_{(p)}} \THC_S(\ell;H\Z_{(p)}), \]
which lie in degrees congruent to $1,2p,$ and $4p-1$ mod $2p^2$, cannot be in the image.
\end{remark}

\begin{proof}[Proof of Theorem~\ref{thm:THCHopf}]
We first note that $\THH(\ell; H\Z_{(p)})$ is a commutative Hopf algebra spectrum over $H\Z_{(p)}$, and the homotopy groups are finitely generated over $H\Z_{(p)}$ in
each degree. The $H\Z_{(p)}$-dual, $\THC_S(\ell;H\Z_{(p)})$, therefore has finitely generated homotopy groups in each degree, and hence the Bockstein spectral sequence
\[
\THC^\ast_S(\ell; H\F_p)[v_0] \Rightarrow \THC^\ast_S(\ell; H\Z_{(p)})_p^\wedge
\]
is a convergent spectral sequence.  Multiplication by $p$ commutes with the comultiplication, and hence this Bockstein spectral sequence is a spectral sequence of Hopf algebras.  In order for the result to be $H\Z_{(p)}$-dual to $\THH(\ell; H\Z_{(p)})$, the differentials are generated by those of the form
\[
d_{i+1}(c_{p^i - 1} x_{2p^2-1}) \doteq v_0^{i+1} c_{p^i}
\]
for $i \geq 0$, where we use the convention that $c_0 = 1$.
\end{proof}

This theorem allows us to compute the cap product
\[
\THC^k_S(\ell;H\Z_{(p)})\otimes\THH_m(\ell;H\Z_{(p)})\to\THH_{m-k}(\ell;H\Z_{(p)}).
\]

\begin{cor}\label{cor:ComputingCaps}
For $k < n$, the cap product satisfies the following formulae
\[
c_k \smallfrown a_n \doteq \binom{n-1}{k} a_{n-k}\text{ and }
c_k \smallfrown b_n \doteq \binom{n-1}{k} b_{n-k}.
\]
\end{cor}

We see that $c_{p^k}$ is $p^{k+1}$ torsion in $\THC^\ast_S(\ell;H\Z_{(p)})$, which means that in $\THC^*_S(\ell)$, $p^k c_{p^k} \neq 0$. However, the Adams spectral sequence for $\THC^\ast_S(\ell)$ suggests that in fact these classes are torsion free.
%

Naturality of the cap product moreover implies that it commutes with the differentials in the Bockstein spectral sequences. We will exploit both of these remarks to compute the differentials and extensions in the remaining spectral sequence.

\section{The last Bockstein spectral sequence and $\rTHH_\ast(\ell)$}\label{sec:SSB4}

The $v_1$-Bockstein spectral sequence is pictured for $p=2$ through dimension $35$ in Figure~\ref{fig:BSS2}. Here multiplication by $2$ preserves the filtration, though we have drawn it as increasing the filtration by $1$ to reduce clutter.

\begin{figure}[ht]
\centering
\includegraphics[width=.9\textwidth]{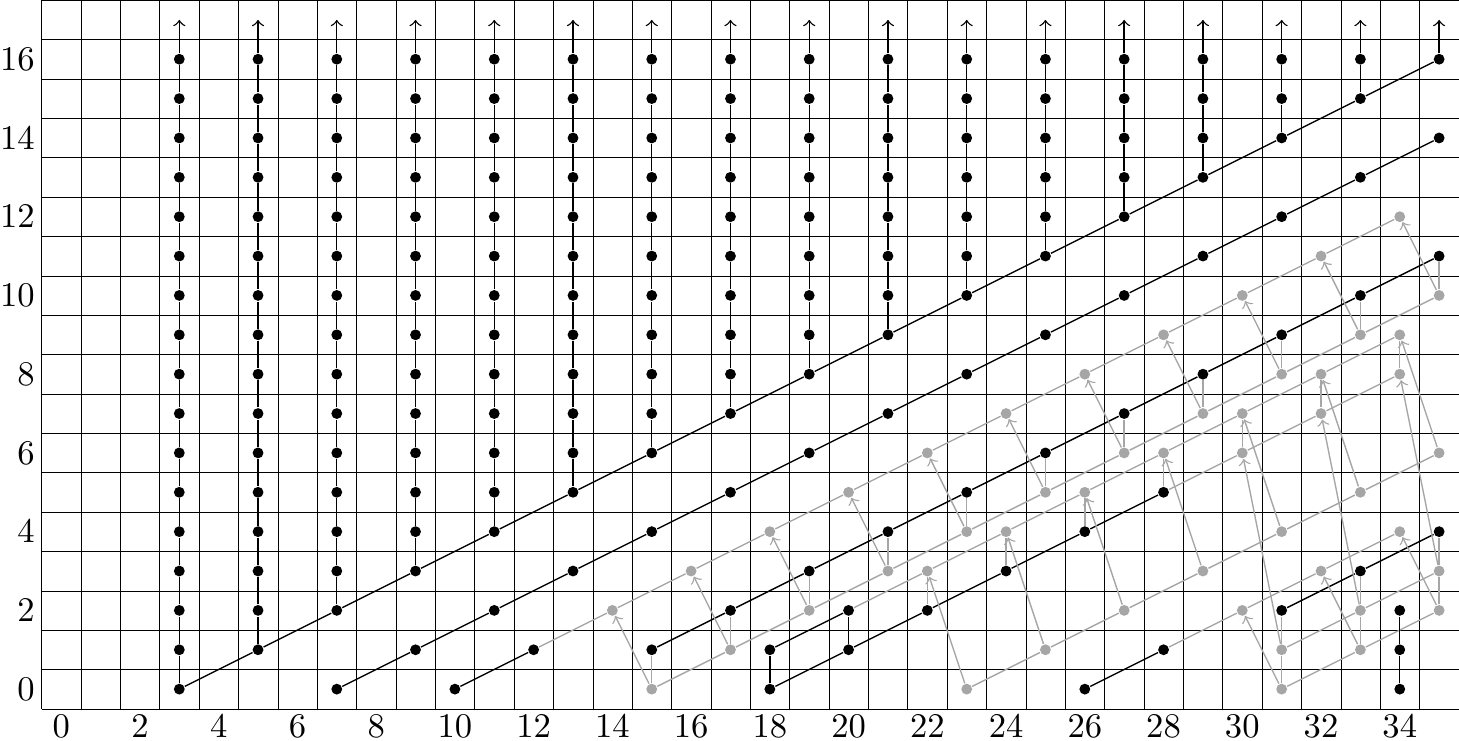}
\caption{The $v_1$-Bockstein spectral sequence converging to $\rTHH_\ast(ku)$} \label{fig:BSS2}
\end{figure}

We can now get all the differentials and extensions in this spectral sequence. Recall that in Section~\ref{sec:HZSS}, we defined torsion elements $a_i=\mu^{i-1}\lambda_2$ and $b_i=\lambda_1 a_i$. Since these and $\lambda_1$ were all of the additive generators of $\rTHH_\ast(\ell;H\Z_{(p)})$, we can immediately conclude the following proposition for degree reasons.

\begin{prop}
The $v_1$ tower on $\lambda_1$ survives to $E_\infty$.
\end{prop}

\begin{lemma}\label{lem:biCycles}
The classes $b_{p^i}$ are permanent cycles for all $i$.
\end{lemma}
\begin{proof}
Corollary~\ref{cor:xn0primesurvive} shows that the class $\lambda_1\lambda_2\mu_2^{p^i-1}$ can be lifted from a class in $\rTHH_{2p^{n+2}+2p-3}(\ell;H\F_p)$ to a class in $\rTHH_{2p^{n+2}+2p-3}(\ell)$. In particular, the reduction modulo $v_1$ of a choice of lift gives $b_{p^i}$, up to multiplication by a $p$-adic unit. Since the Bockstein differentials are the obstruction to lifting a class, we conclude that these all vanish on $b_{p^i}$.
\end{proof}

Capping allows us to bootstrap from this to a much stronger statement. We will use the following classical result for binomial coefficients repeatedly:

\begin{kummer}\cite{Kummer}
The $p$-adic valuation of $\binom{a+b}{a}$ is the number of carries when adding $a$ and $b$ in base $p$.
\end{kummer}

\begin{lemma}\label{lem:AllbiCycles}
The classes $b_i$ are permanent cycles for all $i$.
\end{lemma}
\begin{proof}
Pick $j$ such that $p^j\geq i$, and let $k=p^j-i$. If we consider the $p$-adic expansion of $p^j-1$, then Kummer's Theorem shows that the binomial coefficient in
\[
c_k\smallfrown b_{p^j}=\binom{p^j-1}{k}b_{i}
\]
is a $p$-adic unit. Naturality of the cup product then ensures that for all $m$,
\[
d_m(b_i)\doteq d_m(c_k\smallfrown b_{p^j})=c_k\smallfrown d_m(b_{p^j})=0.\qedhere
\]
\end{proof}

\subsection{The differentials}
The second part of Lemma~\ref{lem:zerocyclic} shows that there are no classes in degree $2p^{n+2}-2$. However, there are a great many classes in the spectral sequence there. All of these classes must be killed, and there is only one pattern of differentials that achieves this.

\begin{theorem}\label{thm:B4diff}
The differentials in Spectral Sequence (\ref{B4}) are determined by
\[
d_{p^n+\ldots+p}(p^{n-1} a_{kp^{n-1}}) \doteq (k-1) v_1^{p^n+\ldots+p} b_{(k-1)p^{n-1}} \]
for all $n \geq 1$ and $k \geq 1$.
\end{theorem}

\begin{proof}
We prove this in three steps. By capping with judiciously chosen classes, we first show that if the differential is as claimed for $k=p$, then it is so for $k\equiv 0$ modulo $p$. Capping down from multiples of $p$ allows us to conclude the differentials for all $k$. We then use induction on $n$ to show the differentials on $p^{n-1}a_{p^n}$. For ease of readability, let $m=p^n+\ldots+p$.

We assume that
\[
d_{m}(p^{n-1}a_{p^n})\doteq v_1^m b_{(p-1)p^{n-1}},
\]
and we will first show that
\[
d_m(p^{n-1} a_{jp^n}) \doteq v_1^m b_{(jp-1)p^{n-1}}
\]
for all $j \geq 1$. Naturality of the cap product with respect to the Bockstein differentials shows that
\[
c_{(j-1)p^n} \smallfrown d_m(p^{n-1} a_{jp^n}) = d_m(c_{(j-1)p^n} \smallfrown p^{n-1} a_{jp^n}).
\]
We know that
\[
c_{(j-1)p^n} \smallfrown p^{n-1} a_{jp^n} = \binom{jp^n-1}{(j-1)p^n}p^{n-1} a_p^n,
\]
and by Kummer's Theorem this binomial coefficient is a $p$-adic unit. Hence
\[
d_m(p^{n-1} a_{jp^n}) \doteq v_1^m b_{(jp-1)p^{n-1}}.
\]
This shows that we have the required differential on $p^{n-1} a_{kp^{n-1}}$ whenever $k\equiv 0$ modulo $p$.

Now suppose $k \not\equiv 0$ modulo $p$. Write
\[
k = jp-k'
\]
where $1 \leq k' \leq p-1$. Then
\[
c_{k'p^{n-1}} \smallfrown a_{jp^n} = \binom{jp^n-1}{k'p^{n-1}} a_{kp^{n-1}}.
\]
This binomial coefficient is a $p$-adic unit, so it follows that
\begin{multline*}
 d_m(p^{n-1} a_{kp^n}) \doteq d_m(c_{k' p^{n-1}} \smallfrown p^{n-1} a_{jp^n}) = c_{k'p^{n-1}} \smallfrown d_m(p^{n-1} a_{jp^n}) \\
\doteq c_{k'p^{n-1}} \smallfrown v_1^m b_{(jp-1)p^{n-1}} = v_1^m
\binom{(jp-1)p^{n-1} -1}{k'p^{n-1}} b_{(k-1)p^{n-1}}.
\end{multline*}
By Kummer's Theorem once more we find that the $p$-adic valuation of $\binom{(jp-1)p^{n-1}-1}{k'p^{n-1}}$ is $\nu_p(k-1)$, so the differentials on $p^{n-1} a_{kp^n}$ for all $k$ follow once we know
\[
d_{m}(p^{n-1}a_{p^n})\doteq v_1^m b_{(p-1)p^{n-1}}.
\]

We prove this by induction on $n$. The case $n=1$ is clear, since there is visibly only one generator in dimension $2p^3-2$, namely $v_1^p b_{p-1}$, so assume that the result is true for all $t<n$. Thus if $\nu_p(r)<n-1$, then the $v_1$-tower of $b_{r}$ is truncated, and therefore $b_r$ will play no further role in the computation.

We first identify all classes in degree $|p^{n-1} a_{p^n}|-1=2p^{n+2}-2$ in the $E_1$-term. The only even generators are the classes $b_i$, so we are looking for those values of $k$ and $i$ such that
\[
|v_1^kb_i|=2(p-1)k+2p^2i+2p-2=2p^{n+2}-2.
\]
Reducing modulo $p-1$ or $p$ shows that $i=(p-1)\hat{\imath}$ and $k=p\hat{k}$ for some integers $\hat{k}$ and $\hat{\imath}$. In other words, we are looking for pairs of positive integers $(\hat{k},\hat{\imath})$ such that
\[
\hat{k}+p\hat{\imath}=p^n+\dots+1.
\]
This breaks the problem into two cases: either $\nu_p(\hat{\imath})=n-1$ or $\nu_p(\hat{\imath})<n-1$. The former corresponds to the unique case $b_{(p-1)p^{n-1}}$ and is the only value we want to show remains by $E_{p^n+\dots+p}$. We want to rule out the latter cases, so assume that $s=\nu_p(\hat{\imath})<n-1$. The inductive hypothesis tells us that $v_1^{p^{s+1}+\dots+p}b_i=0$ on $E_{p^{n}+\dots+p}$. However, we also know that
\[
\hat{k}=p^n+\dots+1-p\hat{\imath}\geq p^s+\dots+1
\]
by the above analysis, so $v_1^kb_i=v_1^{p\hat{k}}b_i=0$, as required.

We therefore conclude that there is only one class remaining in degree $2p^{n+2}-2$: $v_1^m b_{(p-1)p^{n-1}}$. Since $b_{(p-1)p^{n-1}}$ is a cycle, this must be the target of a differential. In degree $2p^{n+2}-1$, there are the various $v_1$-multiples of $p^fa_{p^f}$ for $f<n$ and $p^{n-1}a_{p^{n}}$. By induction, the classes $p^fa_{p^f}$ are permanent cycles (since there are no classes in the degree immediately preceding theirs), and thus we must have the required differential on $p^{n-1}a_{p^n}$.
\end{proof}

An easy corollary of this is that there are gaps in the even dimensional
homotopy of $\rTHH(\ell)$.

\begin{lemma}\label{lem:homotopygaps}
For all $n\geq 0$ and $2\leq k\leq p$, the even dimensional homotopy groups of $\rTHH(\ell)$ are $0$ between degrees $2kp^{n+2}-2p+1$ and $2kp^{n+2}+2p-3$.
\end{lemma}

\begin{proof}
On the $E_\infty$-page, the classes $b_j$ support a $v_1$ tower of length $p^{i}+\dots+p$, where $i=\nu_p(j)+1$. Thus if there are classes in the desired range, then they originate as $v_1$ multiples of classes $b_{kp^n-m}$ for some $m$. The $p$-adic valuation of the subscript is determined by that of $m$, and it is clear that we need only check the largest integers $kp^n-m$ for any given $p$-adic valuation, namely $kp^n-p^m$ for $0 \leq m \leq n$. However, the top $v_1$ multiples of each of these classes lie in the same dimension: $2kp^{n+2}-2p$, proving the result.
\end{proof}

We also note that the proof of Theorem~\ref{thm:B4diff} shows that on the $E_\infty$-page, the $\Z_{(p)}[v_1]$-module generated by $pb_{p^k}$ is isomorphic to the one generated by $b_{p^{k-1}}$.

\subsection{The Torsion Free Extensions}
We adopt the notation in this section that $v_0x$ is the image of multiplication by $p$ on the $E_1$-page of the spectral sequence. We begin with the torsion free part, proving a restatement of Theorem~\ref{thm:v1Divisibility}.

\begin{thm}
We have additive extensions
\[
p\cdot a_1\doteq v_1^p\lambda_1
\]
and for $k\geq 1$,
\[
p\cdot v_0^ka_{p^k}\doteq v_1^{p^{k+1}}v_0^{k-1}a_{p^{k-1}}.
\]
\end{thm}

\begin{proof}
We saw in Lemma~\ref{lem:ExtensionLemma} that
\[
\rTHH_{2p^{i+2}-1}(\ell)=\Z_{(p)}.
\]
However, since the elements $a_{p^j}$ are $p^{j+1}$-torsion, and the differentials above only involve multiples of $a_{p^j}$ up to $v_0^{j-1}a_{p^j}$, in the $2p^{j+2}-1$ stem, there are elements $v_0^ia_{p^i}$ and $v_1^{p^{i+1}+\dots+p^{k+2}}v_0^ka_{p^k}$ for $0\leq k<i$ in this degree. Since the group must be cyclic, we have non-trivial additive extensions, and by the structure of Bockstein spectral sequences, we must have
\[
p\cdot(v_1^{p^{i+1}+\dots+p^{k+2}}v_0^ka_{p^k})\doteq v_1^{p^{i+1}+\dots+p^{k+1}}v_0^{k-1}a_{p^{k-1}}.
\]
Comparing powers of $v_1$ provides the desired result.
\end{proof}

\subsection{The Torsion Extensions}
That there are extensions between the torsion patterns is clear: the basic differentials all arise on $p$-multiples of the classes $a_{p^n}$. By naturality, the targets of these differentials must be linked by similar multiplications by $p$, and this essentially gives Theorem~\ref{thm:LastTorsionExt}.

Carefully proving the extensions in the torsion patterns is much harder. We begin by isolating repeating patterns in the torsion. For $n\geq 0$ and $1\leq k\leq p-1$, let $T_{n,k}$ be the submodule of $\rTHH_\ast(\ell)$ generated by all classes $b_i$, $kp^n\leq i\leq (k+1)p^n-1$, with degrees shifted down so that the lowest class is in degree 0. Lemma~\ref{lem:homotopygaps} shows that the torsion of $\rTHH_\ast(\ell)$ is the direct sum of shifts of these modules.

\begin{thm}\label{thm:IsomPieces}
The submodule $T_{n,k}$ is independent of $k$, $0\leq k\leq p-1$.
\end{thm}
\begin{proof}
This is another capping argument. We will cap down from the case $k=p-1$. To get the lower torsion submodules, we will cap with classes $c_{jp^n}$, $1\leq j\leq p-2$. The generators of the examined submodule are those $b_i$ with $(p-1)p^n\leq i\leq p^{n+1}-1$. When we cap with $c_{jp^n}$, we get
\[
c_{jp^n}\smallfrown b_i=\binom{i-1}{jp^n}b_{i-jp^n}.
\]
However, for all $i$ in the desired range, the $p$-adic expansion of $i-1$ begins with at most $(p-2)p^n$, this binomial coefficient is a $p$-adic unit and hence an isomorphism onto.
\end{proof}

\begin{cor}
The torsion submodule of $\THH_\ast(\ell)$ splits as a direct sum
\[
\bigoplus_{n\geq 0}\bigoplus_{k=1}^{p-1}\Sigma^{2kp^{n+2}+2(p-1)}T_{n,k}
\]
\end{cor}

It remains only to determine the structure of $T_{n,k}$ for some $k$. We will identify some extensions in $T_{n,p-1}$ and use these to determine $T_{n,1}$ completely.

We now exploit the first part of Lemma~\ref{lem:zerocyclic}: $\rTHH_{2p^{n+3}-2p}(\ell)$ is a cyclic group. On the $E_\infty$-page of the spectral sequence, there are only the elements
\[
v_1^{(p^{k+1}+\dots+p)-1}b_{p^{n+1}-p^{k}}, \quad 0\leq k\leq n,
\]
since the other $v_1$ multiples of the intermediate elements $b_j$ are all killed by differentials. This means that there must be extensions linking these elements.

\begin{thm}\label{thm:LastTorsionExt}
There are choices of lifts of the generators in this spectral sequence so that we have hidden additive extensions
\[
p\cdot b_{p^{n+1}-p^k}=v_0b_{p^{n+1}-p^k}+v_1^{p^{k+2}}b_{p^{n+1}-p^{k+1}},
\]
where $v_0b_{p^{n+1}-p^k}$ is the image of $p$ times $b_{p^{n+1}-p^k}$ on the $E_1$-page, and where $0\leq k\leq n$.
\end{thm}

\begin{proof}
We show the extensions by increasing degree. We first note that the convergence of the Bockstein spectral sequence ensures that we can find lifts of the classes $v_0b_{p^{n+1}-p^k}$ which have $v_1$ order exactly what is seen on the $E_\infty$-page. Since the order of all elements of higher filtration in the same degree is larger than that of $v_0b_{p^{n+1}-p^k}$, this choice is unique. Moreover, this implies that $p^m$ times this lift is a lift of $v_0^{m+1}b_{p^{n+1}-p^k}$ which has the same $v_1$ order as the image in $E_\infty$. We choose this unit so that $p(b_{p^{n+1}-p^k})=v_0b_{p^{n+1}-p^k}$ modulo elements of higher filtration.

Since we must have the extensions at the end of the $v_1$ torsion patterns generated by the $b_{p^{n+1}-p^k}$, there must be extensions linking the generating classes. The argument is now one of decreasing induction on $k$. There is only one class of higher filtration than $b_{p^{n+1}-p^{n-1}}$, namely $v_1^{p^{n+1}}b_{p^{n+1}-p^n}$. We must therefore have an extension
\[
p\cdot b_{p^{n+1}-p^{n-1}}-v_0b_{p^{n+1}-p^{n-1}}\doteq v_1^{p^{n+1}}b_{p^{n+1}-p^n}.
\]
By changing both $b_{p^{n+1}-p^{n-1}}$ and $v_0b_{p^{n+1}-p^{n-1}}$ by the same unit, we can ensure actual equality.

The inductive step is identical. There must be a non-trivial extension from $b_{p^{n+1}-p^{k}}$ to the subgroup generated by elements of higher filtration which realizes this element as the generator of a cyclic group. The subgroup of elements of higher filtration is generated by $v_1^{p^{k+2}}b_{p^{n+1}-p^{k+1}}$, so we must have an extension of the form
\[
p\cdot b_{p^{n+1}-p^k}-v_0b_{p^{n+1}-p^k}\doteq v_1^{p^{k+2}}b_{p^{n+1}-p^{k+1}}.
\]
Simultaneously changing the lifts of $b_{p^{n+1}-p^k}$ and $v_0b_{p^{n+1}-p^k}$ by a unit allows us to produce an equality.
\end{proof}

Combined with Theorem~\ref{thm:IsomPieces}, this gives hidden extensions of the form
\[
p\cdot b_{2p^{n}-p^k} \doteq v_0b_{2p^{n}-p^k}+v_1^{p^{k+2}}b_{2p^{n}-p^{k+1}}.
\]
We will use this to find $T_{n,1}$. After changing the choices of generators by multiplication by units if necessary, we find the following.

\begin{thm}\label{thm:TorsionExtensionsinFull}
There are lifts of the generators $v_0^i b_m$ to $\Sigma^{2p^{n+2}-2(p-1)} T_{n,1}$ for $p^n \leq m \leq 2p^n-1$ as follows. Let $k = \nu_p(m)$ and let $k'=\nu_p(m-(p-1)p^k)$. Then
\[ p \cdot b_m = \begin{cases}
v_0 b_m + v_1^{p^{k+2}} v_0^{k'-k-1} b_{m-(p-1)p^k} & \text{if $k'-k-1 \geq 0$} \\
v_0 b_m & \text{otherwise.}
\end{cases} \]
\end{thm}

\begin{proof}
We will show this by capping down from $b_{2p^n-p^k}$. We find that
\[ c_{2p^n-p^k-m} \smallfrown b_{2p^n-p^k} = \binom{2p^n-p^k-1}{2p^n-p^k-m} b_m \]
is a $p$-adic unit times $b_m$.

Naturality of the cap product ensures that when we pull back $p \cdot b_{2p^n-p^k}$ via capping with $c_{2p^n-p^k-m}$ we get $p \cdot b_m$. Hence it is enough to determine
\[ c_{2p^n-p^k-m} \smallfrown v_1^{p^{k+2}} b_{2p^n-p^{k+1}} = v_1^{p^{k+2}} \binom{2p^n-p^{k+1}-1}{2p^n-p^k-m} b_{m-(p-1)p^k}.\]
By Kummer's Theorem, we find that the $p$-adic valuation of this binomial coefficient is $k'-k-1$ if $k'-k-1 \geq 0$ and $0$ if $k'=k$. Hence the result follows by noting that if $k'=k$ then $v_1^{p^{k+2}} b_{m-(p-1)p^k} = 0$.
\end{proof}

This result completes our analysis of $T_{n,1}$. We can now provide a dictionary linking $T_{n,1}$ with the module $T_n$ defined in Section~\ref{sec:prelim}. We define a bijection between strings of length at most $n$ and $v_0$ multiples of classes $b_i$, $p^n\leq i\leq 2p^n-1$, via
\[
a_1\dots a_k\underbrace{0\dots 0}_j \longleftrightarrow v_0^j b_{p^n+a_1p^{n-1}+\dots+a_kp^{n-k}}.
\]
The restriction on the lengths of the strings reflects both the fact that we only consider classes in a prescribed range and the fact that the order of a class $b_i$ is $p^{\nu_p(i)+1}$. The previous theorem then shows that all of the relations from Section~\ref{sec:prelim} are satisfied.

\section{Topological Hochschild homology of $ko$}\label{sec:ko}
We can follow the same program as for $\THH(ku)$ to calculate $\THH_*(ko)_{(2)}$. As a starting point, Rognes and the first author \cite{AnRo} used the B\"okstedt spectral sequence to conclude that
\[
\THH_\ast(ko;H\F_2)=E(\lambda_1',\lambda_2) \otimes P(\mu),
\]
where $|\lambda_1'|=5$, $|\lambda_2|=7$, and $|\mu|=8$.  Here the class $\lambda_1'$ is represented by $\sigma \xi_1^4$ \cite[Thm 6.2]{AnRo}.

This serves as the starting point for a chain of spectral sequences, just as before. In this case, however, we have an $\eta$-Bockstein spectral sequence in addition to the four spectral sequences analogous to the $ku$ case.

\begin{prop}
There is a bigraded spectral sequence of algebras
\begin{equation} \label{eq:E_1ko}
E_1=\THH_\ast(ko;ku)[\eta] \Rightarrow \THH_\ast(ko).
\end{equation}
\end{prop}
This spectral sequence is the Bockstein spectral sequence associated to the cofiber sequence
\[
\Sigma ko\xrightarrow{\eta} ko\to ku.
\]
Since $\eta^3=0$ in $ko_\ast$, the spectral sequence has a horizontal vanishing line at filtration $3$, and $E_4=E_\infty$.

Just as before, this spectral sequence is essentially an Adams spectral sequence: this is the $ku$-based Adams spectral sequence in the category of $ko$-modules. There is a slight difference between this case and the earlier ones: the $E_1$-page is not given by the appropriate minimal resolution (since $(ku_\ast, \pi_\ast(ku\sma_{ko} ku))$ is only a Hopf algebroid). This difficulty is reflected in the multiplicative structure: $v_1$ and $\eta$ anti-commute. However, from the $E_2$-page and beyond, the Adams and Bockstein spectral sequences coincide. The classes in $\THH_\ast(ko;ku)$ have bidegree $(\ast,0)$, while $\eta$ has bidegree $(1,1)$, and the differentials are Adams type.

\subsection{Statement of results}
The homotopy groups of $\rTHH(ko)$ sit as an extension of two parts, one from the torsion free part of $\rTHH_*(ko;ku)$ (though this part will contain torsion) and the other from the torsion part of $\rTHH_*(ko;ku)$.

Define a $ko_*$-module $F^{ko}$ as follows. Additively,
\[
F^{ko} = \bigoplus_{i \neq 2^n-2} \Sigma^{4i} \Z[\eta]/(2\eta,\eta^2)
\oplus \bigoplus_{n \geq 1} \Sigma^{4(2^n-2)} \Z.
\]
Multiplication by $v_1^2$ sends the $\Z$ in degree $4i$ isomorphically to the $\Z$ in degree $4i+4$, except when $i=2^n-2$, in which case multiplication by $v_1^2$ sends the $\Z$ to $2 \Z$. (Since $v_1^2 \not \in ko_*$ we should instead say that multiplication by $2v_1^2 \in ko_*$ acts as multiplication by $2$ when $i \neq 2^n-2$ and as multiplication by $4$ when $i=2^n-2$.) In all figures that follow, multiplication by $v_1^2$ will be denoted by a dashed line.

Next we define the ``torsion pieces'' $\tilde{T}_n$ and $D\tilde{T}_n$. Let
\[
\tilde{T}_n = \Z[v_1^2]/(2^n, 2^{n-1} v_1^2, \ldots, (v_1^2)^{2^n-1}).
\]

Let $D$ denote the $\Z[v_1^2]$-module given as the quotient
\[
\Z/{2^\infty}[v_1^2]\to \Z/{2^{\infty}}[v_1^{\pm 2}]\to D\to 0.
\]
This is the dualizing object for $\Z[v_1^2]$-modules. Let
\[
D\tilde{T}_n=Hom_{\Z[v_1^2]}(\tilde{T}_n,D)
\]
denote the dual of $\tilde{T}_n$. Since $\tilde{T}_n$ is positively graded, starting in dimension $0$, $D\tilde{T}_n$ is negatively graded with top class in dimension $0$.

These modules occur in dual pairs, grouped according to torsion patterns from $\rTHH_\ast(ko;ku)$. Let
\[
T_n^{ko}=\bigoplus_{k=1}^{2^{n-1}-1}\left(\Sigma^{16k}\tilde{T}_{\nu(k)+1}\oplus\Sigma^{2^{n+3}-16k-10}D\tilde{T}_{\nu(k)+1}\right)\oplus
\left(\tilde{T}_n\oplus\Sigma^{2^{n+3}-10}D\tilde{T}_n\right)
\]

\begin{theorem} \label{thm:THHko}
As a $ko_*$-module, $\rTHH_\ast(ko)$ sits in a short exact sequence
\[
0\to\Sigma^5F^{ko} \to\rTHH_\ast(ko)\to \bigoplus_{n>0} \Sigma^{2^{n+3}+4}T_n^{ko}\to 0.
\]
The extension is completely determined by the requirement that twice the generator of lowest degree in $\Sigma^{2^{n+4}-6}D\tilde{T}_n\subset \Sigma^{2^{n+3}+4}T_n^{ko}$ is the unique non-zero element in $\Sigma^5F^{ko}$ in that degree.
\end{theorem}

\begin{cor}
On $\rTHH(ko)$, $\eta^2$ acts trivially.
\end{cor}

\begin{proof}
This is clear because $\eta^2$ acts trivially on $\tilde{F}$ and $\eta$ acts trivially on each $\tilde{T}_n$ and $D\tilde{T}_n$.
\end{proof}

There is another $ko$-module in which $\eta^2$ is $0$: the connective, self-conjugate $K$-theory spectrum $kc=ko\wedge C(\eta^2)$. This is an $E_\infty$-ring spectrum (arising as the connective cover of $KU^{h\Z}$ where $\Z$ acts as $\psi^{-1}$ through its quotient $\Z/2$), and for modules over this spectrum, $v_1^2$-multiplication is a well-defined operation, since this is an element of $\pi_\ast kc$ \cite{Bo90}. Since $\eta^2$ acts as zero in $\rTHH(ko)$, we present the following conjecture.

\begin{conj}
The $ko$-module summand $\rTHH(ko)$ admits the structure of a module over $kc$.
\end{conj}

Figure~\ref{fig:THHko} shows the homotopy of $\rTHH(ko)$ through degree $34$, while Figure~\ref{fig:THHkoTorsion} shows the torsion in $\rTHH_*(ko)$ arising from $\Sigma^{32}T_2^{ko}$ and $\Sigma^{68}T_3^{ko}$ between degree $36$ and $92$. In Figure~\ref{fig:THHkoTorsion} the vertical arrows indicate extensions connecting the classes to $\Sigma^5 F^{ko}$.

\begin{figure}[ht]
\centering
\includegraphics[width=.9\textwidth]{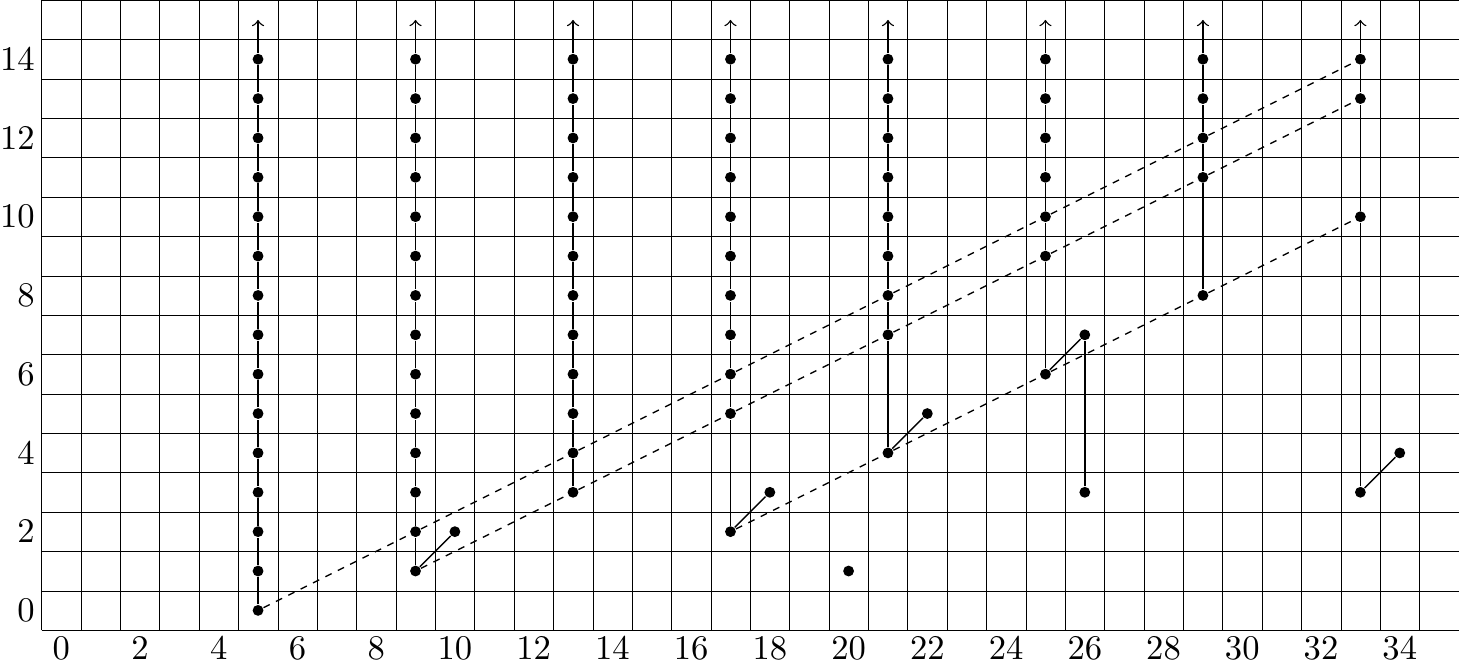}
\caption{$\rTHH_*(ko)$ through degree $34$} \label{fig:THHko}
\end{figure}

\begin{figure}[ht]
\centering
\includegraphics[width=.9\textwidth]{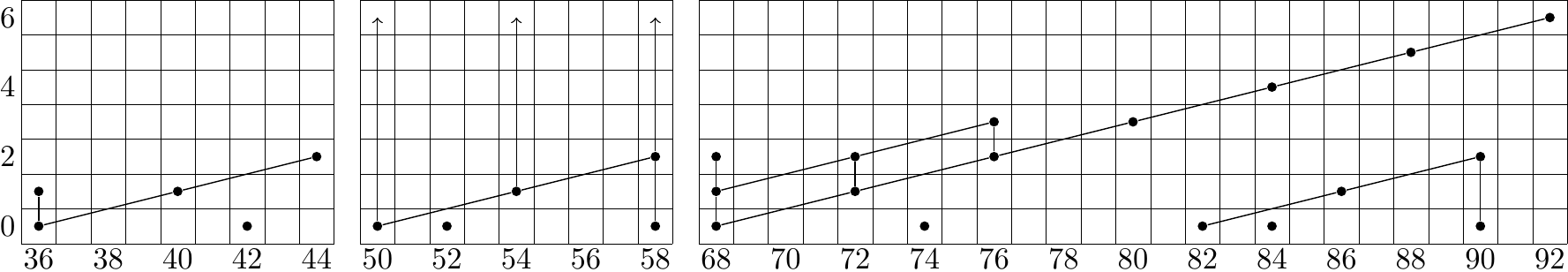}
\caption{The torsion between degrees $36$ and $92$} \label{fig:THHkoTorsion}
\end{figure}

\subsection{Computing $\rTHH_\ast(ko;ku)$}
This computation is similar to that of $\THH(ku)$, and we omit the proofs. We remark, however, that $\THH_S(ku)$ acts on $\THH(ko;ku)$, allowing a faithful mirroring of the proofs. Let $a_i'$ denote a lift of $\lambda_2\mu^{i-1}$ from $\rTHH_\ast(ko;H\F_2)$ to $\rTHH_\ast(ko;H\Z)$ and let $b_i'=\lambda_1'a_i'$. We therefore have $|a_i'|=8i-1$ and $|b_i'|=8i+4$. These classes play the roles of the classes $a_i$ and $b_i$. We have the following results:

\begin{theorem}\label{thm:v1Divisibilityko}
The torsion free summand of $\rTHH_*(ko;ku)$ is $F'\cdot\lambda_1'$, where $F'$ is the $ku_*$-module
\[
F'=ku_*\left[\frac{v_1^{2^{k+1}-3}}{2^k};\,k\geq 1\right]\subset ku_*\otimes\Q.
\]
\end{theorem}

The torsion is also similar. Define $T_0'=ku_*/(2,v_1)$, and define $T_n'$ recursively by gluing together two copies of $T_{n-1}'$ along a $v_1$-tower of length $2^{n+2}-3$. Alternatively, define $T_n'$ as $\Sigma^{-2} v_1 T_n \subset \Sigma^{-2} T_n$.

\begin{theorem} \label{thm:Torsionko}
The torsion summand of $\rTHH_*(ko;ku)$ is, as a $ku_*$-module, $2$-locally isomorphic to the following direct sum:
\[
\bigoplus_{n \geq 0} \Sigma^{2^{n+3} + 4}T_n'.
\]
\end{theorem}

These results are closely related to those of Section~\ref{sec:SSB4}. The natural ring map from $ko$ to $ku$ induces maps of the four Bockstein spectral sequences, and this will allow us to relate the resulting computations. As initial input, the B\"okstedt spectral sequence shows that the natural map from $\THH_\ast(ko;H\F_2)$ to $\THH_\ast(ku;H\F_2)$ sends $\lambda_1'$ to $0$ and $\lambda_2$ and $\mu$ to themselves. This means that $a_i'$, being represented by $\lambda_2\mu^{i-1}$, maps to $a_i$. The $ku_\ast$-module structure then forces $\lambda_1'$ and $2^ka_{2^k}'$ in $\THH_\ast(ko;ku)$ to map to $v_1\lambda_1$ and $2^ka_{2^k}$ respectively in $\THH_\ast(ku)$. Naturality of the Bockstein spectral sequences then ensures that $b_i'$ maps to $v_1b_i$.

\begin{prop}
The homotopy of $\rTHH(ko;ku)$ sits as the $ku_\ast$-submodule of $\rTHH_\ast(ku)$ generated by $v_1\lambda_1$, the classes $2^k a_{2^k}$, and $v_1$ times the torsion patterns $T_n$.
\end{prop}

Considering the effect of inverting $v_1$ allows us to conclude the following Corollary.

\begin{cor}
The canonical map $\THH(KO;KU)\to\THH(KU)$ is a weak equivalence.
\end{cor}

In particular, a homotopy-fixed point / Galois descent argument allows us to determine $\THH(KO)$.

\begin{cor}
As a $KO$-module,
\[
\THH(KO)=KO\vee\Sigma KO_\Q.
\]
\end{cor}

The homotopy of $\rTHH(ko;ku)$ is depicted through degree $35$ in Figure~\ref{fig:THHkoku}.

\begin{figure}[ht]
\centering
\includegraphics[width=.9\textwidth]{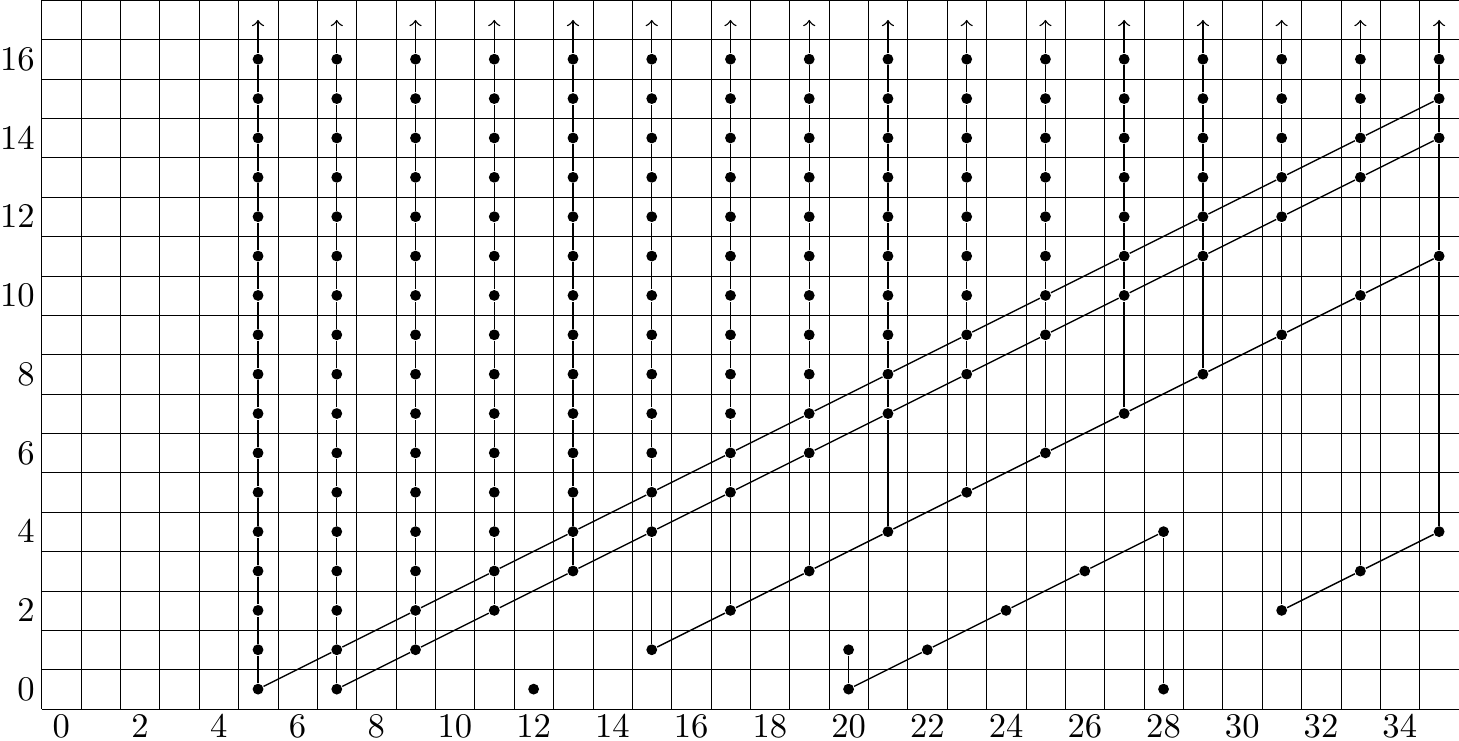}
\caption{The homotopy of $\rTHH(ko;ku)$}
\label{fig:THHkoku}
\end{figure}

\subsection{The $\eta$-Bockstein spectral sequence}
We first review the $\eta$-Bockstein spectral sequence $E_1(ku)=ku_*[\eta] \Rightarrow ko_*$. In this case we have
\[
d_1(v_1)=2\eta,\, d_1(v_1^2)=0,\text{ and }d_3(v_1^2)=\eta^3.
\]
In particular, $E_2=\Z[v_1^2,\eta]/2\eta$ has infinite $\eta$-towers off of every power of $v_1^2$, and $E_4$ equals $E_\infty$. There are no extensions, and $E_\infty$ is the homotopy of $ko$.

Let us write Spectral Sequence (\ref{eq:E_1ko}) as
\[
E_1=\Sigma^5 E_1(F')\oplus\bigoplus_{n \geq 1}\Sigma^{2^{n+3}+4}E_1(T_n'),
\]
where $E_1(F')=F'[\eta]$ and $E_1(T_n')=T_n'[\eta]$.

\begin{prop}\label{prop:etasplitting}
The splitting
\[
E_1=\Sigma^5 E_1(F')\oplus\bigoplus_{n \geq 1}\Sigma^{2^{n+3}+4}E_1(T_n')
\]
is a direct sum of differential graded modules over $E_1(ku)$. Moreover, there are no differentials connecting the summands $E_1(T_n')$ and $E_1(T_m')$ for $n\neq m$.
\end{prop}

\begin{proof}
Let $T'=\bigoplus\Sigma^{2^{n+3}+4}T_n'$. Then for all $x\in T'$, the degree of $x$ is even, while for every $y\in \Sigma^5F'$, the degree of $y$ is odd. Since $d_1$-differentials change parity, there are no $d_1$-differentials connecting $E_1(T')$ and $E_1(\Sigma^5 F)$. (Similarly, there are no possible $d_3$-differentials connecting these summands.)

For the second part, we again argue by degrees. The analysis of the even dimensional homotopy of $\rTHH(ko;ku)$ shows that between dimensions $2^{n+3}-2$ and $2^{n+3}+2$, $\rTHH_{even}(ko;ku)=0$. In particular, any differential connecting $\Sigma^{2^{n+3}+4}T_n'$ and $\Sigma^{2^{n+2}+4}T_{n-1}'$ must be a $d_{>3}$.
\end{proof}

\begin{thm}
In $E_1(\Sigma^5 F')$, we have $d_1$-differentials determined by
\[
d_1 \left( \frac{v_1^{2^{n+1}-3}}{2^n} \lambda_1' \right) = \eta \frac{v_1^{2^{n+1}-4}}{2^{n-1}} \lambda_1'.
\]
\end{thm}
\begin{proof}
For degree reasons, $\lambda_1'$ is a permanent cycle. The result follows immediately from the $E_1(ku)$-differential graded module structure.
\end{proof}

Note that this leaves the $\eta$-towers on classes
\[
\frac{v_1^{2^{n+1}-2+2i}}{2^{n}} \lambda_1'
\]
untruncated for each $n \geq 1$ and $0 \leq i \leq 2^{n}-2$. These classes link the modules $E_2(\Sigma^5 F')$ and $E_2(T')$, and, as we shall see, are permanent cycles. Some of these are easy to see, however.

\begin{cor}\label{cor:d2cycles}
The classes $\frac{v_1^{2^{n+1}-2}}{2^n} \lambda_1'$ are permanent cycles, and so there are no differentials from the torsion free summand to the torsion summands.
\end{cor}

\begin{proof}
The first part is an immediate degree check. The class $\frac{v_1^{2^{n+1}-2}}{2^n}\lambda_1'$ is in degree $2^{n+2}+1$. By the previous theorem, the closest $\eta$-torsion free class of smaller degree is $b_{2^{n-1}-1}'$, which is in degree $2^{n+2}-4$. Since the spectral sequence collapses at $E_4$, we cannot have any differentials originating on our class.

As a consequence, since $v_1^2$-multiplication commutes with $d_1$ and $d_2$-differentials, we learn that all of the $\eta$-torsion free classes in the torsion free part, the previously described $v_1^2$-multiples of $\frac{v_1^{2^{n+1}-2}}{2^n}\lambda_1'$, are $d_1$- and $d_2$-cycles. Proposition~\ref{prop:etasplitting} shows that the only possible differentials connecting the torsion and torsion free summands are $d_2$-differentials, so we conclude that there are no differentials from torsion free classes to torsion ones.
\end{proof}

\begin{thm}\label{thm:koDifferentials}
Using the module structure over the spectral sequence for $E_\ast(ku)$, the differentials in the torsion summands are determined by the following:
\begin{eqnarray*}
d_1 ( b_{a2^k}' ) & = & \eta \frac{a-1}{2} v_1^{2^{k+2}-1} b_{(a-1)2^k}' \text{ for $a \neq 1$ odd, and} \\
d_2 ( b_{2^n}' ) & = & \eta^2 \frac{v_1^{2^{n+2}-2}}{2^{n+1}} \lambda_1' .
\end{eqnarray*}
\end{thm}

\begin{proof}
We prove that $d_1$ or $d_2$ on $b_i'$ is as claimed by induction on $i$. Write $i$ as $a2^k$ with $a$ odd. We begin by listing all the possible differentials on $b_{a2^k}'$ by considering all classes in degree $|b_{a2^k}'|-1$.

We first consider $d_1$ and $d_3$-differentials. By Proposition~\ref{prop:etasplitting}, we know that these all take place within a single torsion summand. Here the analysis of the structure of the torsion summands allows us to quickly enumerate classes. For each class $b_i'$, the only classes in degree $|b_i'|-2$ and $|b_i'|-4$ are the appropriate $v_1$-multiples of those classes which arise in the hidden multiplication-by-2 extensions seen in Theorem~\ref{thm:TorsionExtensionsinFull}. Since $p=2$, essentially every class had non-trivial extensions and the analysis is substantially simplified.

For each $m \geq 1$ such that $a+1\equiv 0$ modulo $2^m$ and $a \neq 2^m-1$, there is a possible $d_1$-differential
\[
d_1(b_{a2^k}') = \eta \frac{a+1-2^m}{2^m} v_1^{(2^m-1)2^{k+2}-1} b_{(a+1-2^m)2^k}'.
\]
and a possible $d_3$-differential
\[
d_3(b_{a2^k}') = \eta^3 \frac{a+1-2^m}{2^m} v_1^{(2^m-1)2^{k+2}-2} b_{(a+1-2^m)2^k}'.
\]
The possible targets are also the only classes in the appropriate degree whose $v_1^2$-order is less than or equal to that of $b_i'$. We remark in passing that the coefficients seen here are, up to a $2$-adic unit, the same integers as the powers of $v_0$ seen in Theorem~\ref{thm:TorsionExtensionsinFull}.

The $d_2$-differentials connect the torsion and torsion free summands. A counting check shows that if $n$ is such that $2^n \leq a2^k \leq (2^{n+1}-1)$, there is a possible $d_2$ differential
\[
d_2(b_{a2^k}') = \eta^2 \frac{v_1^{a 2^{k+2}-2}}{2^{n+1}}\lambda_1'.
\]

We now prove the result by induction on $i$. The base case of $i=1$ is immediate from the collapse of the spectral sequence at $E_4$. Now assume that the differentials are as listed for all $j<i$. We will first show that there is only one possible differential on $b_i'$ whose target is non-zero modulo the differentials implied by the induction hypothesis, and we will then show that $b_i'$ cannot be a permanent cycle, even if corrected by $v_1^2$-divisible terms. This will show that up to a different basis, the differentials are as described.

By the final part of Proposition~\ref{prop:etasplitting}, if $a=1$, so $i=2^n$ for some $n$, there is always only one possible non-trivial differential:
\[
d_2(b_{2^n}') = \eta^2 \frac{v_1^{2^{n+2}-2}}{2^{n+1}} \lambda_1'.
\]

Now assume that $a>1$. If $m \geq 1$ with $a \neq 2^m-1$, then by the induction hypothesis, the class
\[
\eta^3 \frac{a-2^m+1}{2^m} v_1^{(2^m-1)2^{k+2}-2} b_{(a+1-2^m)2^k}'
\]
supports a $d_1$ or $d_2$-differential, since $b_{(a+1-2^m)2^k}'$ does and these differentials commute with $v_1^2$-multiplication. Thus there cannot be any $d_3$-differentials.

For $d_1$-differentials, if $m>1$ with $a \neq 2^m-1$, we find by induction that
\[
d_1(v_1^{(2^{m-1}-1)2^{k+2}} b_{(a-2^{m-1}+1)2^k}')=\eta \frac{a-2^m+1}{2^m} v_1^{(2^m-1)2^{k+2}-1} b_{(a-2^m+1)k}'.
\]
Similarly, we have
\[
d_2(v_1^{a2^{k+2}-2^{n+2}} b_{2^n}')=\eta^2 \frac{v_1^{a2^{k+2}-2}}{2^{n+1}} \lambda_1'.
\]
We therefore conclude that by adding $v_1$-divisible elements, we may assume without loss of generality that either $b_{a2^k}'$ is a permanent cycle or the differential is as described in the theorem.

Now suppose $b_{a2^k}'$ is an infinite cycle. Then it follows from the $E_*(ku)$-module structure that the top nonzero $v_1$-power on $b_{a2^k}'$, $v_1^{2^{k+2}-4} b_{a2^k}'$, is an infinite cycle as well. Hence $\eta^3 v_1^{2^{k+2}-4} b_{a2^k}'$ must be a boundary.

We consider all possible classes in degree $|\eta^3 v_1^{2^{k+2}-4} b_{a2^k}'|+1=(a+1)2^{k+3}$. The $v_1$-towers on $b_j'$ for $j<a2^k$ have already been accounted for by induction. Corollary~\ref{cor:d2cycles} shows that there can be no differentials from the torsion free summands, so we only need to consider the $v_1$-towers on $b_j'$ for $a2^k<j<(a+1)2^k$, for degree reasons. However, the $v_1$-towers on these classes, together with the $v_1$-tower on $b_{a2^k}'$, generate a copy of $T_k'$ starting in degree $a2^{k+3}+4$ and ending in degree $(a+1)2^{k+3}-4$. In particular, none of these can support a differential truncating $\eta$ on $v_1^{2^{k+2}-4} b_{a2^k}'$, and we conclude that $b_{a2^k}'$ must support the nontrivial differential.
\end{proof}

We present the $E_2$-page through dimension $35$ as Figure~\ref{fig:THHkoE2}. We remark that though we have drawn $v_1$ and $v_0$ in filtration $1$, in the Bockstein spectral sequence they have filtration $0$. Thus the differentials drawn are simply $d_2$-differentials. We remark also that classes drawn in black have filtration zero or $1$, while those drawn in gray have filtration at least two.

\begin{figure}[ht]
\centering
\includegraphics[width=.9\textwidth]{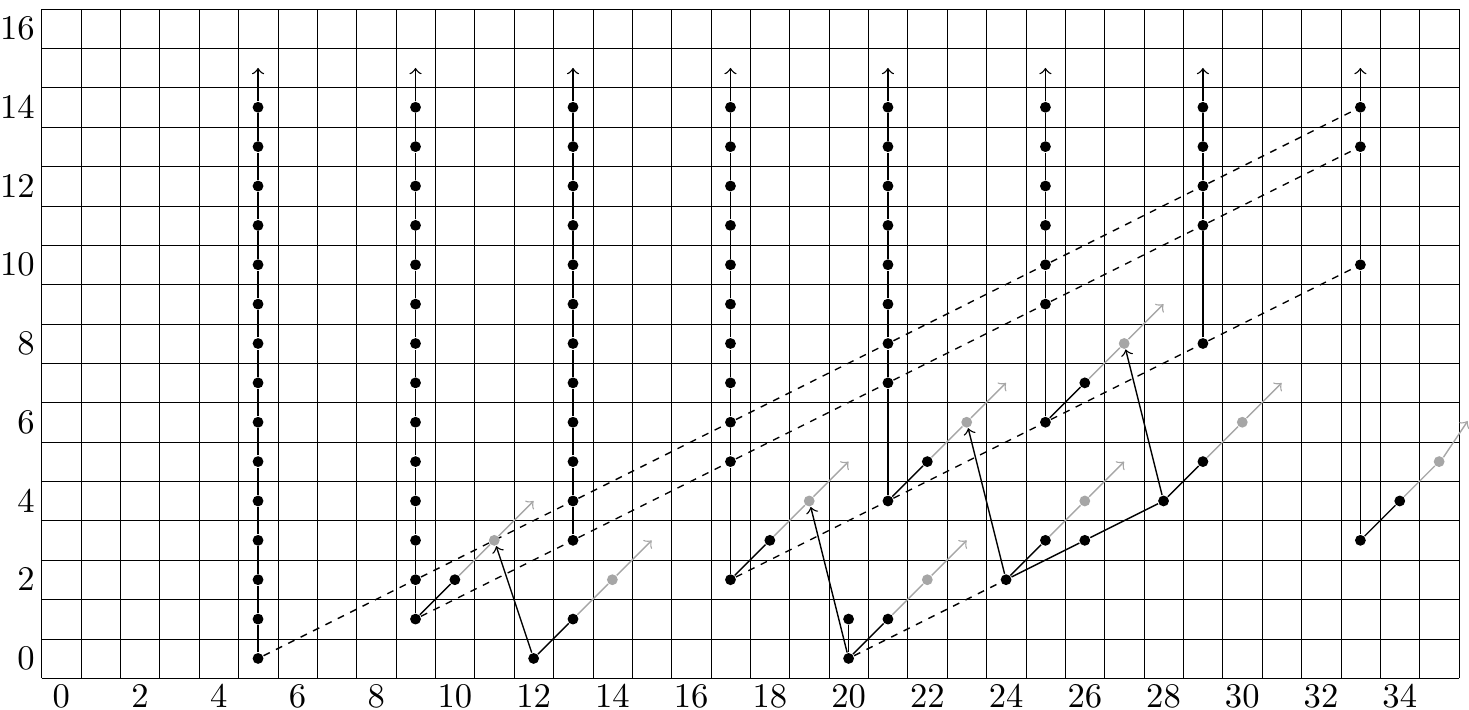}
\caption{The $\eta$-Bockstein $E_2$-page} \label{fig:THHkoE2}
\end{figure}

Note that the $E_\infty$ term consists of a direct sum of $F^{ko}$ with a sum of copies of $\tilde{T}_k$ and $D\tilde{T}_k$ for each $T'_n$, so this proves Theorem~\ref{thm:THHko} up to extensions. In particular, all classes in the spectral sequence are either $\eta$ or $\eta^2$ torsion. We remark that while it is in general tedious to name the generators of the summands of $T^{ko}_n$, the lowest degree class in the copy of $D\tilde{T}_n$ in $T^{ko}_n$ is represented by $v_1^{2^{n+1}-1}b_{2^n}'$. This element and its $v_1^2$-multiples are the sources of the hidden extensions.

\subsection{Resolving the Extensions}
The results about differentials show that $\rTHH_\ast(ko)$ is some extension of the direct sum of all of the torsion modules with $\Sigma^5 F^{ko}$. It remains only to solve this extension problem. By the structure of Bockstein spectral sequences, the target of an exotic multiplication by $2$ or $v_1^4$ must be $\eta$-divisible.

\begin{lemma}
The only possible additive extensions are hidden multiplications-by-$2$ connecting $\Sigma^{2^{n+4}-6} D\tilde{T}_n\subset T_n'$ and $F^{ko}\cdot\lambda_1'$ for $n \geq 1$.
\end{lemma}

\begin{proof}
The only $\eta$-divisible classes are in degrees congruent to $2$ modulo $4$. This rules out any exotic extensions originating on the summands of the form $\tilde{T}_{k}$ from any $T_n^{ko}$, and these therefore occur as direct summands of the homotopy. Similarly, there are no exotic extensions originating on the submodule $F^{ko}\cdot\lambda_1'$, since the target of these must be $\eta$-divisible.

We argue the remaining cases by considering the $v_1^4$-orders of the possible sources and targets of exotic multiplications by $2$. We consider the summands of the form $D\tilde{T}_{\nu(k)+1}$ in $T_n^{ko}$. For degree reasons, the $v_1^4$-order of every element in $D\tilde{T}_{\nu(k)+1}$ is strictly less than the $v_1^4$-order of the possible targets of exotic multiplications by $2$. Thus, if there were an exotic multiplication by $2$ on an element $a$, then we would necessarily have a complementary exotic multiplication by $v_1^4$. By replacing $a$ with its largest non-trivial $v_1^4$-multiple, we may assume without loss of generality that $v_1^4 a=0$. Then if $2\cdot a=\eta\cdot b$, we have
\[
v_1^4\cdot\eta\cdot b=v_1^4\cdot 2\cdot a=2\cdot v_1^4\cdot a=2\cdot v_1^4\cdot \eta\cdot b,
\]
since $v_1^4\cdot\eta\cdot b$ is the only possible non-trivial target of any hidden extensions in the same degree as $v_1^4 a$. Since $v_1^4$-multiplication is faithful in these degrees, we conclude there can be no such extension.
\end{proof}

We remark that the arguments employed in the previous lemma do not apply to elements from the summand $\Sigma^{2^{n+4}-6}D\tilde{T}_n\subset T_n'$, and in fact, here we see extensions.

\begin{thm}
There are hidden multiplications-by-$2$
\[
2\cdot v_1^{2^{n+1}-1+2k}b_{2^n}'=\eta v_1^{2^{n+1}+2k}\frac{v_1^{2^{n+2}-2}}{2^{n+1}}\lambda_1',
\]
for $0\leq k\leq 2^n-2$ and $n\geq 1$.
\end{thm}

\begin{proof}
The $\eta$-Bockstein spectral sequence used above arises by iterating the cofiber sequences
\[
\Sigma ko\xrightarrow{\eta} ko\xrightarrow{\rho} ku.
\]
In truth, this amounts to a simplification of bookkeeping: by considering infinitely many copies of the cofiber sequence, indexed by the multiples of the degree of $\eta$, we can remember $\eta$-torsion information. It considerably simplifies the exposition for this argument to consider instead the singly graded spectral sequence arising from the exact couple given by applying $\pi_\ast$ to the above cofiber sequence.

Let $D_\ast=\rTHH_\ast(ko)$ and $E_\ast=\rTHH_\ast(ko;ku)$. The maps in the exact couple are given by multiplication-by-$\eta$:
\[
\eta_\ast\colon D_\ast\to D_{\ast+1},
\]
the reduction modulo $\eta$:
\[
\rho_\ast\colon D_\ast\to E_\ast,
\]
and the connecting homomorphism:
\[
\partial\colon E_\ast\to D_{\ast-2}.
\]
The maps in the bigraded $\eta$-Bockstein spectral sequence also all arise from the maps $\eta_\ast$, $\rho_\ast$, and the connecting map $\partial$, and thus they are the same as for the singly graded spectral sequence. Our earlier analysis of the bigraded Bockstein spectral sequence therefore gives a complete analysis of the differentials in the singly graded Bockstein spectral sequence.

We consider the class $a=b'_{3\cdot 2^{n-1}}\in \rTHH_\ast(ko;ku)$. By Theorem~\ref{thm:koDifferentials}, this class supports a $d_1$-differential of the form
\[
d_1(a)=\eta v_1^{2^{n+1}-1} b'_{2^n}.
\]
By the same theorem, $2a$ supports a $d_2$-differential:
\[
d_2(2a)=\eta^2 v_1^{2^{n+1}}\frac{v_1^{2^{n+2}-2}}{2^{n+1}}\lambda_1'.
\]
Tracing these statements back to the exact couple proves our result. Consider the element $\partial(a)\in\rTHH_\ast(ko)$. By definition, this is an $\eta$-torsion lift of $v_1^{2^{n+1}-1} b'_{2^n}$ in $\rTHH_\ast(ko;ku)$. Since $d_1(2a)=0$, we learn that $2\partial(a)=\partial(2a)$ in $\rTHH_\ast(ko)$ is in the image of multiplication-by-$\eta$. The $d_2$-differential amounts to division by $\eta$ followed by application of $\rho_\ast$, and this tells us that $\eta^{-1}2\partial(a)$ is detected by $v_1^{2^{n+1}}\tfrac{v_1^{2^{n+2}-2}}{2^{n+1}}\lambda_1'$ in $\rTHH_\ast(ko;ku)$. We therefore conclude that in $\rTHH_\ast(ko)$,
\[
2 v_1^{2^{n+1}-1} b'_{2^n}=2\partial(a)=\eta\big(\eta^{-1}2\partial(a)\big)=\eta v_1^{2^{n+1}}\frac{v_1^{2^{n+2}-2}}{2^{n+1}}\lambda_1'.
\]

This argument applies {\em{mutatis mutandis}} to the classes $v_1^{2}a$, and multiplication by powers of $v_1^4$ then completes the proof.
\end{proof}


\bibliographystyle{plain}
\bibliography{b}

\end{document}